\documentclass[a4paper, oneside]{amsart}
\usepackage[utf8]{inputenc}

\pdfoutput=1

\usepackage[english]{babel}
\usepackage{amsthm,amsmath,amssymb}
\usepackage[usenames,dvipsnames,svgnames,table]{xcolor}
\usepackage[all]{xy}
\usepackage{xy,mathrsfs,pstricks,color,verbatim,numprint}
\usepackage{color,xcolor,transparent,subfigure,soul}  
\usepackage{tikz,ltablex,booktabs}
\usetikzlibrary{positioning,arrows.meta,decorations.text,calc}
\usetikzlibrary{shapes.multipart}

\usepackage{verbatim}

\input xy
\xyoption{all}
\newdir{ >}{{}*!/-5pt/\dir{>}}
\usepackage[a4paper]{geometry}
\usepackage[utf8]{inputenc}
\usepackage[english]{babel}
\usepackage{amsthm,amsmath,amssymb, amsfonts, amsxtra}
\usepackage[usenames,dvipsnames,svgnames,table]{xcolor}
\usepackage[all]{xy}
\usepackage{xy,mathrsfs,pstricks,color,verbatim,numprint}
\usepackage{color,xcolor,transparent,subfigure,soul}  
\usepackage{tikz,ltablex,booktabs, tikz-cd}
\usepackage{verbatim}
\usepackage{ upgreek }

\newcommand{\imono}[1]{ \;\xymatrix{  \ar@{>->}^{#1}[r] &  \\} }
\newcommand{\iepi}[1]{ \;\xymatrix{  \ar@{->>}^{#1}[r] &  \\} }
\newcommand{\mono}{ \;\xymatrix{  \ar@{>->}[r] &  \\} }
\newcommand{\epi}{ \xymatrix{   \ar@{->>}[r] &  \\} }

\def\rep{\mbox{rep}\,}
\def\add{\mbox{add}\,}
\newcommand{\bD}{\mathbb{D}}
\newcommand{\bE}{\mathbb{E}}
\newcommand{\bI}{\mathbb{I}}
\newcommand{\bV}{\mathbb{V}}
\newcommand{\bW}{\mathbb{W}}

\newcommand{\A}{\mathcal{A}}

\newcommand{\D}{\mathcal{D}}
\newcommand{\E}{\mathcal{E}}
\newcommand{\I}{\mathcal{I}}

\newcommand{\V}{\mathcal{V}}
\newcommand{\W}{\mathcal{W}}

\newcommand{\cF}{\mathcal{F}}

\renewcommand{\S}{\mathcal{S}}

\def\Ex{\mbox{\bf Ex}}
\def\Bf{\mbox{\bf Bf}}
\def\Wex{\mbox{\bf Wex}}
\def\LW{\mbox{\bf LW}}
\def\RW{\mbox{\bf RW}}
\def\Bf{\mbox{\bf BiFun}}
\def\Cbf{\mbox{\bf CBiFun}}
\def\Ab{\mbox{\bf Ab}}
\def\Bim{\mbox{\bf Bim}}
\def\Cbim{\mbox{\bf Cbim}}

\newcommand{\ca}{{\mathcal A}}
\newcommand{\cA}{{\mathcal A}}

\newcommand{\cc}{{\mathcal C}}

\newcommand{\ce}{{\mathcal E}}

\newcommand{\ct}{{\mathcal T}}

\newcommand{\cW}{{\mathcal W}}

\newcommand{\Eff}{\operatorname{\bf{Eff}}}
\newcommand{\coh}{\operatorname{\bf{coh}}}
\newcommand{\defff}{\operatorname{\bf{def}}\nolimits}

\newcommand{\Mod}{\operatorname{\bf{Mod}}\nolimits}
\renewcommand{\mod}{\operatorname{\bf{mod}}\nolimits}
\newcommand{\fp}{\operatorname{\bf{fp}}\nolimits}
\newcommand{\fg}{\operatorname{\bf{fg}}\nolimits}

\theoremstyle{definition} 
\newtheorem{theorem}{Theorem}[section]
\newtheorem{proposition}[theorem]{Proposition}
\newtheorem{lemma}[theorem]{Lemma}
\newtheorem{cor}[theorem]{Corollary}
\newtheorem{remark}[theorem]{Remark}
\newtheorem{example}[theorem]{Example}
\newtheorem{definition}[theorem]{Definition}

\theoremstyle{definition} 

\def\End{\mbox{End}}
\def\Hom{\mbox{Hom}}
\def\Ext{\mbox{Ext}}

\def\Mod{\mbox{Mod\,}}
\def\rep{\mbox{rep}\,}

\def\Ab{\mbox{Ab}}
\def\B{\mbox{B}}
\def\Set{\mbox{Set}}

\usepackage{graphicx}
\usepackage{float}

\title{On the lattices of exact and weakly exact structures}
\author[R.-L. Baillargeon]{Rose-Line Baillargeon}
\author[T. Br\"ustle]{Thomas Br\"ustle}
\author[M. Gorsky]{Mikhail Gorsky}
\author[S. Hassoun]{Souheila Hassoun}

\address{Universit\'e de Sherbrooke, 2500, boul. de l'Universit\'e, Sherbrooke, Qu\'ebec J1K 2R1}
\address{Fakult\"at f\"ur Mathematik, Universit\"at Wien, Oskar-Morgenstern-Platz 1, 1090 Wien, Austria}
\email{Rose-Line.Baillargeon@USherbrooke.ca}
\email{Thomas.Brustle@USherbrooke.ca}
\email{mikhail.gorskii@univie.ac.at}
\email{Souheila.Hassoun@USherbrooke.ca}
\usepackage{graphicx}

\begin{document}
\maketitle

\bigskip

\begin{abstract}
 We initiate in this article the study of \emph{weakly exact structures}, a generalisation of Quillen exact structures. We introduce weak counterparts of one-sided exact structures and show that a left and a right weakly exact structure generate a weakly exact structure. We further define weakly extriangulated structures on an additive category and characterize weakly exact structures among them.\\ 
 We investigate when these structures on $\A$ form lattices. We prove that the lattice of substructures of a weakly extriangulated structure is isomorphic to the lattice of topologizing subcategories of a certain functor category. In the weakly idempotent complete case, we characterise the lattice of all weakly exact structures and we prove the existence of a unique maximal weakly exact structure.\\
 We study in detail the situation when $\A$ is additively finite, giving a module-theoretic characterization of closed sub-bifunctors of $\Ext^1$ among all additive sub-bifunctors.
\end{abstract}

\section{Introduction}
In this work we are studying \emph{all} the additive sub-bifunctors of the first extension functor $Ext$ on an additive category. They are associated to what we introduce as \emph{the weakly exact structures}, which generalize  Quillen's exact structures.

 In the first papers on relative homological algebra, a mix of structures has been considered that correspond to an exact structure:  on one hand classes of morphisms satisfying certain properties (``h.f.class''), on the other hand certain (``closed'') subfunctors of $\Ext$. The authors considered also a weaker notion, an {\em f.class}, which omits the condition on admissible monics and epics to be closed under composition.
 This weaker notion corresponds to an additive subfunctor of Ext. It has been studied more recently in the work of Fu, Guil Asensio, Herzog and Torrecillas \cite{FGHT}, which extended the theory of approximation in the relative homological algebra to the setup of morphisms rather than objects. They  demonstrated the need to study the more general notion of {\em f.class} by considering examples such as the Auslander–Reiten phantom morphisms.
\cite{FGHT} work in the context of a given exact category $(\A,\E)$, and consider links between ideals of morphisms and additive sub-bifunctors of the extension functor $Ext_{\E}$ associated to $\E.$ 
 This work has been further extended by Breaz and Modoi \cite{BM} to the context of extension-closed subcategories $\mathcal{A}$ of a triangulated category $\mathcal{T}$ and restrictions of sub-bifunctors of $\mathcal{T}(-, -[1])$ on $\mathcal{A}.$
 \\
 
 The ``stand alone'' concept of an exact structure as a class of short exact sequences in an additive category $\A$ satisfying certain axioms has been laid out by Quillen in \cite{Qu73}. A more concise version of these axioms was formulated by Keller in \cite{Ke}, see also \cite{GR}. 
The comparison to sub-bifunctors of $\Ext^1$ has been re-considered in \cite{AS1} and then in \cite{DRSS}, with applications to exact structures originating from one-point extensions, a special case of exact structures associated with bimodule problems in \cite{BrHi}. However, the lack of a unique maximum extension-functor for arbitrary additive categories was a limiting factor in these studies. 
Only later Rump \cite{Ru11} showed that any additive category admits a unique maximal exact structure $\E_{\max}$. 
In \cite{BHLR} a study of the family of all exact structures $\Ex(\A)$ on an additive category $\A$ was initiated. The existence of a unique maximum exact structure allows to turn $\Ex(\A)$ into a complete bounded lattice. On the side of bifunctors, this amounts to studying all closed sub-bifunctors of a unique maximum bifunctor $\bE_{\max}$ which corresponds to the exact structure $\E_{\max}$. It is natural, on the bifunctor side, to extend the study to {\em all} additive sub-bifunctors, which in turn raises the question to which structure of exact sequences they correspond.\\

In this work we introduce the notion of a  {\em weakly exact structure} on an additive category $\A$. It provides a conceptualization of the notion of f.classes studied in \cite{Buch} and the notion of additive subfunctors of an extension functor $\E$ studied in \cite{AS1, DRSS, FGHT}.
We establish the existence of a unique maximal weakly
exact structure provided the additive category $\A$ is weakly idempotent complete. This in turn allows to
show that all the weakly exact structures on $\A$ form a lattice. 

We introduce in Section \ref{section:exact structures} the class $\Wex(\A)$ of all weakly exact structures on an additive category $\A$.
It turns out that, despite the fact that the classes of admissible monics and admissible epics of weakly exact structures are not necessarily closed under compositions, some of the properties of exact structures are still valid, in particular, every weakly exact structure satisfies Quillen's obscure axiom, see Proposition \ref{prop:obscure-axiom}.
Similar to exact structures, it is sometimes beneficial to dissect the set of axioms into two parts, leading to the notion of left and right weakly exact structures. We show that any pair of a left and a right weakly exact structures gives rise to a weakly exact structure, and that all such structures arise in that way.
\medskip

It is known and proved in \cite{Cr11}, that the stable exact structure $\E_{sta}$ forms the maximal exact structure on any weakly idempotent complete category. 
In this work we apply these results to deduce that a unique maximal weakly exact structure exists on any weakly idempotent complete category, and is given by the stable short exact sequences. 
We also consider the interval $\Wex(\E_{max}):=[\E_{min},\E_{max}] \subseteq \Wex(\A)$ and we study the weakly exact structures that are included in the unique maximal exact structure $\E_{max}$.
Given a weakly exact structure $\W$ on $\A$, constructing the group $\bW$ of $\W-$extensions 
yields a map $\Phi$ 
to category of bifunctors from $\A$ to abelian groups:
\begin{align*}
\Phi : \Wex(\A) & \longrightarrow \Bf(\A)\\ \W & \longmapsto \bW=\Ext^1_{\W}(-,-).
\end{align*}

This function $\Phi$ induces lattice isomorphisms
\[\begin{array}{ccc}
 \Wex(\E_{max}) & \longleftrightarrow & \Bf(\bE_{max}) \\
 \cup &   & \cup\\
\Ex(\A) & \longleftrightarrow & \Cbf(\A)
\end{array}\]
where $\Cbf(\A)$ denotes the subclass of closed sub-bifunctors of $\bE_{max}$.\\
Note that $\Ex(\A)$ is {\em not} a {\em sublattice} of $\Wex(\E_{max})$, even if it is a subposet: the join operations we consider on these sets are different, as we illustrate by an example in Section \ref{section:example}.
\bigskip

When the underlying category $\A$ is additively finite and Krull-Schmidt, it is known that the lattice $\Ex(\A)$ is boolean, with each object $\bE(S)$ determined by the choice of a set $S$ of Auslander-Reiten sequences.
The larger lattice $\Wex(\A)$ however is not boolean, and it is interesting to characterise the members of $\Ex(\A)$ in module-theoretic terms, that is, describe the closed sub-bimodules of $\bE_{max}$. We show that, when viewed as bimodules over the Auslander algebra of $\A$, elements in $\Ex(\A)$ can be characterized as follows: 
For every set  $S$ of Auslander-Reiten sequences, the closed bimodule $\bE(S)$ of $\bE_{max}$ introduced above  is the maximal submodule of $\bE_{max}$ whose socle is $S$.\\

In order to find a general and simultaneous way to give proofs of various statements concerning exact and triangulated categories at the same time, Nakaoka and Palu \cite{NakaokaPalu1} studied additive bifunctors $\ce: \ca^{op} \times \ca \to \Ab$ equipped with certain extra data called a \emph{realization}. 
They found a set of axioms on triples consisting of an additive category, a bifunctor and a realization that unifies the axioms of exact and of triangulated categories. They called such structures \emph{extriangulated}. Extensions in exact categories are realized by ``admissible'' kernel-cokernel pairs. In an extriangulated category this role is played by pairs of composable morphisms $f, g$ where $f$ is a \emph{weak kernel} of $g$ and $g$ is a \emph{weak cokernel} of $f$.
Moreover, Nakaoka and Palu characterized all triples that define exact structures, in other words, closed additive sub-bifunctors of $\mathbb{E}_{max}$.  Hershend, Liu and Nakaoka \cite{HLN} introduced \emph{$n-$exangulated} structures and proved that the choice of a $1-$exangulated structure on an additive category is equivalent to the choice of an extriangulated structure. The set of axioms of $1-$exangulated structures is slightly different from that of extriangulated categories. In Section 5, we consider $1-$exangulated categories with one of the axioms removed.
We prove that such \emph{weakly $1-$exangulated}, or \emph{weakly extriangulated} structures naturally generalize weakly exact structures we defined earlier. We also show that \emph{almost exact structures} on extension closed subcategories of triangulated categories, which were considered by Breaz and Modoi in \cite{BM}, are weakly extriangulated. 

For a finite-dimensional algebra $\Lambda,$ Buan \cite{Buan} studied closed sub-bifunctors of the bifunctor $\Ext^1_{\Lambda}$ on the category $\mod \Lambda$. He proved that they correspond to certain Serre subcategories of the category of finitely presented additive functors $(\mod \Lambda)^{op} \to \Ab$ (i.e. of \emph{finitely presented} modules over $\mod \Lambda$), defined as categories of \emph{contravariant defects} in works of Auslander. This result  was later extended to exact structures on additive categories in \cite{En18}, see also \cite{FG}. 
We note that the definition of contravariant defects naturally extends to the setting of weakly exact structures. Ogawa \cite{Ogawa} defined contravariant defects in the setting of extriangulated categories, and we further extend this notion to the framework of weakly extriangulated categories. By adapting arguments of Ogawa and Enomoto \cite{Enomoto3}, we prove that the category of defects of a weakly extriangulated structure on an additive category $\ca$ is \emph{topologizing} (in the sense of Rosenberg \cite{Rosenberg}) in the category $\coh(\ca)$ of coherent right $\ca-$modules. That means that it is closed under subquotients and finite coproducts. 

Given a weakly extriangulated structure, all its substructures are uniquely characterized by their categories of defects, and each topologizing subcategory of a given category of defects defines a weakly extriangulated substructure. Weakly extriangulated substructures of a weakly exact structure are necessarily weakly exact. Thus, whenever we know that an additive category $\ca$ admits a unique maximal weakly exact structure, we can classify all weakly exact structures on $\ca$ in terms of topologizing structures in a certain abelian category. As explained above, this covers all weakly idempotent complete additive categories.

Topologizing subcategories of an abelian category form a lattice. Topologizing subcategories of the (not necessarily abelian) category of defects of a weakly extriangulated structure on $\ca$ also form a lattice, which is an interval in the lattice of all topologizing subcategories of $\coh(\ca)$. Note that Serre subcategories form a subposet, but not a sublattice of this lattice. Weakly extriangulated substructrures of a weakly extriangulated structure also form a natural lattice, extending the lattice of weakly exact structues. We establish, in the last section of this work, lattice isomorphisms between these  lattices.\\

\section{Acknowledgements}
The authors would like to thank Charles Paquette for his thoughtful comments on the PhD thesis of the fourth author. The authors would also like to thanks Shiping Liu and Hiroyuki Nakaoka for helpful discussions contributing to this version of the work.\\ Most of this work was done while the first author was supported by an NSERC USRA grant. The second author was supported by Bishop's University and NSERC of Canada, and the fourth author acknowledges support from the "thésards étoiles" scholarship of ISM for outstanding PhD candidates. This work was completed during the third author's participation at the Junior Trimester Program "New Trends in Representation Theory" at the Hausdorff Institute for Mathematics in Bonn. He is very grateful to the Institute for the perfect working conditions. This work is a part of a project that has received funding from the European Research Council (ERC) under the European Union’s Horizon 2020 research and innovation programme (grant agreement No. 101001159).
\\
This article is an abridged version of the more detailed version available on arXiv \cite{BBGH}.

\section{Weakly exact and exact structures}\label{section:exact structures}

\subsection{Exact structures}


We fix an additive category $\mathcal{A}$ throughout this section. The notion of short exact sequence is specified to be  a  kernel-cokernel pair $(i, d)$, that is, a pair of composable morphims such that $i$ is kernel of $d$ and $d$ is cokernel of $i$.
An exact structure on $\A$ is then given by a class $\mathcal{E}$ of kernel-cokernel pairs on $\mathcal{A}$ satisfying certain axioms which we recall below. We call {\em admissible monic} a morphism $i$ for which there exists a morphism $d$ such that $(i,d) \in \mathcal{E}$. An {\em admissible epic} is defined dually. Note that admissible monics, admissible epics and admissible short exact squences are referred to as inflations, deflations and conflations in \cite{GR}, respectively. 
We depict an admissible monic by  $\mono$
and an admissible epic by $\epi$. The pair $(i,d) \in \mathcal{E}$ is referred to as {\em admissible short exact sequence}, or {\em short exact sequence in $\E$.} 

\begin{definition}\label{Quillen def}
An {\em exact structure} $\mathcal{E}$ on $\A$ is a class of kernel-cokernel pairs $(i, d)$ in $\A$ which is closed under isomorphisms and satisfies the following axioms:

\begin{itemize}
\item[(E0)] For all objects  $A$ in $\mathcal{A}$ the identity $1_A$ is an admissible monic;

\item[(E0)$^{op}$] For all objects  $A$ in $\mathcal{A}$ the identity $1_A$ is an admissible epic; 

\item[(E1)]  The class of admissible monics is closed under composition

\item[(E1)$^{op}$] The class of admissible epics is closed under composition;
\item[(E2)] 
 The push-out of an admissible monic $i: A \mono B$ along an arbitrary morphism $t: A \to C$ exists and yields an admissible monic $s_C$:
$$\xymatrix{
A \; \ar[d]_{t} 
\ar@{>->}[r]^{i}  \ar@{}[dr]|{\text{PO}} 
& B\ar[d]^{s_B}\\
C \; \ar@{>->}[r]_{s_C} & S.}$$

\item[(E2)$^{op}$]  
The pull-back of an admissible epic $h$ along an arbitrary morphism $t$ exists and yields an admissible epic $p_B$
$$\xymatrix{
P \; \ar[d]_{p_A} 
\ar@{ ->>}[r]^{p_B} \ar@{}[dr]|{\text{PB}}  & B\ar[d]^{t}\\
A \; \ar@{->>}[r]_{h} & C.}$$
\end{itemize}

An {\em exact category} is a pair $(\mathcal{A}, \mathcal{E})$ consisting of an additive category $\mathcal{A}$ and an exact structure $\mathcal{E}$ on $\mathcal{A}$. 
\end{definition}

We denote by $(\Ex({\A}), \subseteq)$ the poset of exact structures $\E$ on $\A$, where the partial order is given by containment $\E' \subseteq \E$. Note that $\Ex({\A})$ need not actually form a set, 
but by abuse of language, we still use the term poset when $\Ex({\A})$ is a class. 
The poset $(\Ex({\A}), \subseteq)$ always contains a unique minimal element, the {\em split exact structure} $\E_{min}$ which is formed by all split exact sequences, that is,  sequences isomorphic to  (see \cite[Lemma 2.7]{Bu}):
\[\xymatrix{
A \, \ar@{ >->}[r]^{\footnotesize\begin{bmatrix}
   1 \\
    0 \\
\end{bmatrix}
} & A\oplus B \ar@{ ->>}[r]^{[0 \; 1]} & B \\
}
\]
Moreover, every additive category admits a unique maximal exact structure $\E_{max}$, see
\cite[Corollary 2]{Ru11}.
When the category $\A$ is abelian, then $\E_{max}$ is formed by all short exact sequences in $\A$. The construction is more subtle for other classes of additive categories, we refer to \cite[Section 2.4]{BHLR} for a more detailed discussion.

\subsection{Example}\label{A3}
Consider the category $\A = \rep Q$ of representations of  the quiver
\[ Q : \qquad \xymatrix{ 1 \ar[r]  & 2   & 3 \ar[l]}\]
Then the Hasse diagram of the poset of exact structures $\Ex(\A)$ has the shape of a \emph{cube} (see \cite[Example 4.2]{BHLR} for detailed description of the different exact structures on $\A$):

\begin{center}
\begin{tikzpicture}[scale=3]

 \node (a) at (0,0,0) {$\E_{min}$};

 \node (b) at (1,0,0) {$\E_{1}$};
 \node (c) at (1,1,0) {$\E_{1,3,5}$};
 \node (d) at (0,1,0) {$\E_{3}$};

 \node (e) at (0,0,1) {$\E_{2}$};
 \node (f) at (1,0,1) {$\E_{1,2}$};
 \node (g) at (1,1,1) {$\E_{max}$};
 \node (h) at (0,1,1) {$\E_{2,3,4}$};
 
 \draw [-latex] (a) -- (e);
 \draw [-latex] (e) -- (h);
 \draw [-latex] (h) -- (g);
 
 \draw [-latex] (a) -- (b);
 \draw [-latex] (b) -- (c);
 \draw [-latex] (c) -- (g);
 
 \draw [-latex] (a) -- (d);
 \draw [-latex] (d) -- (h);
 \draw [-latex] (d) -- (c);
 
 \draw [-latex] (b) -- (f);
 \draw [-latex] (e) -- (f);
 \draw [-latex] (f) -- (g);
\end{tikzpicture}
\end{center}

\subsection{Weakly exact structures}

\begin{definition}\label{set of weakly exact}
Let $\A$ be an additive category. We define a {\em weakly exact structure} $\W$ on $\A$ as a class of kernel-cokernel pairs $(i, d)$ in $\A$ which is closed under isomorphisms and direct sums, and satisfies the axioms $(E0)$, $(E0)^{op}$,$(E2)$ and $(E2)^{op}$ of Definition \ref{Quillen def}.
\end{definition}

This definition provides of the conceptualization  of  subfunctors of Ext as studied in \cite{FGHT} in the context of exact categories. We denote by  $(\Wex(\A), \subseteq)$ the poset of all weakly exact structures on $\A$, ordered by containment.

\begin{lemma}
$\Ex(\A)$ is a subclass of $\Wex(\A)$. 
\end{lemma}
\begin{proof}
Only the direct sum condition needs to be verified. But this is always satisfied for exact structures, by \cite[Proposition 2.9]{Bu}.
\end{proof}

We now state some of the properties for exact structures that also hold for weakly exact structures, see \cite[Proposition 2.12]{Bu}:
\begin{lemma}\label{prop-2.12}
Let $\W$ be a weakly exact structure and let $i$ and $i'$ be admissible monics of $\W$ forming the rows of a commutative square:
 \[
    \xymatrix{
    A  \ar@{ >->}[r]^{i} \ar@{->}[d]^{f}
    & B \ar@{->}[d]^{f'}    \\
    A'  \ar@{ >->}[r]^{i'} & B'   }
    \]
Then the following statements are equivalent:\\
(i) The square is a push-out. \\
(ii) $\xymatrix{
A \, \ar@{ >->}[r]^{\footnotesize\begin{bmatrix}
   i \\
    -f \\
\end{bmatrix}
} & B\oplus A' \ar@{ ->>}[r]^{\; \; \;\;  [f' \; i']} & B' \\
}
$ is a short exact sequence belonging to $\W$.\\
(iii) The square is both a push-out and a pull-back.\\
(iv) There exists a commutative diagram with rows being admissible short exact sequences in $\W$:
 \[
    \xymatrix{
    A \; \ar@{>->}[r]^{i} \ar@{->}[d]^{f}
    & B \ar@{->}[d]^{f'}  \ar@{->>}[r]^{p} & C \ar@{->}[d]^{1_{C}}  \\
    A' \ar@{ >->}[r]^{i'} & B' \ar@{->>}[r]^{p'} & C }
\]
\end{lemma}

Commutative squares that are both a pushout and a pullback are called \emph{bicartesian} squares.

\begin{lemma} \label{factorization}
(Compare \cite[Lemma 3.1]{Bu}) Let $\cW$ be a weakly exact structure on $\ca.$ For any morphism of admissible short exact sequences
\[\xymatrix@R=0.6cm{*+++{A} \ar@{>->}[d] \ar@{->}[r] &*+++{B} \ar@{->>}[r] \ar@{->}[d] &*+++{C} \ar@{->}[d]\\*+++{A'} \ar@{>->}[r] &*+++{B'}  \ar@{->>}[r] &*+++{C'}}\]
in $\cW$, there exists a commutative diagram
\[\xymatrix@R=0.6cm{*+++{A} \ar@{->}[d] \ar@{>->}[r] &*+++{B} \ar@{->>}[r] \ar@{->}[d]&*+++{C} \ar@{=}[d]\\*+++{A'} \ar@{>->}[r]  \ar@{=}[d] &*+++{E} \ar@{->}[d] \ar@{->>}[r] &*+++{C} \ar@{->}[d]\\*+++{A'} \ar@{>->}[r] &*+++{B'}  \ar@{->>}[r] &*+++{C',}}\]
where the middle row is also an admissible short exact sequence in $\cW$ and the top left and bottom right squares are bicartesian.
\end{lemma}

\color{black}
In \cite[Lemma 5]{FGHT}, a weaker version of Quillen's obscure axiom is established in the context of weakly exact structures. In fact, the full version is valid in this context, as can be shown using \cite[Proposition 2.16]{Bu} :
\begin{proposition} \label{prop:obscure-axiom}
{\bf (Quillen's obscure axiom for weakly exact structures)} 
Let $\W$ be a weakly exact structure on an additive category $\A$.
\begin{enumerate}
    \item Consider morphisms $\xymatrix{ A \ar@{ ->}[r]^{i} & B \ar@{->}[r]^{j} & C  }$ in $\A$, where $i$ has a cokernel and $ji$ is an admissible monic of $\W$. Then $i$ is also an admissible monic of $\W$. 
    \item Consider morphisms $\xymatrix{ X \ar@{ ->}[r]^{f} & Y \ar@{->}[r]^{g} & Z  }$  in $\A$, where $g$ has a kernel and $gf$ is an admissible epic of $\W$. Then $g$ is also an admissible epic of $\W$.
\end{enumerate}
\end{proposition}

Note that the split exact structure $\E_{min}$ forms the unique minimal element of the  poset $(\Wex({\A}), \subseteq)$.


\subsection{The left and right weakly exact structures}\label{section:left-and-right-weakly}

We define \emph{left weakly exact structures} and \emph{right weakly exact structures}, generalising the left and right exact structures introduced in \cite[Definition 3.1]{BC} and studied in \cite{HR20,Ru11}.

\begin{definition}\label{right weakly} 
A right weakly exact structure on $\A$ is a class of kernels $I$ which is closed under isomorphisms and satisfies the following properties: 
\begin{itemize}
\item[(Id)]  For all objects  $X$ in $\mathcal{A}$ the identity $1_X$ and the zero monomorphism $0 \longrightarrow X$ are in $I$. 

\item[(P)] 
 The push-out of $f: X \longrightarrow Y \in I$ along an arbitrary morphism $h: X \longrightarrow X'$ exists and yields a morphism $f'\in I$:
$$\xymatrix{
X \; \ar[d]_{h} 
\ar@{->}[r]^{f}  \ar@{}[dr]|{\text{PO}} 
& Y\ar[d]^{h'}\\
X' \; \ar@{->}[r]_{f'} & Y'}$$
\item[(Q)] Given $\xymatrix{ A \ar@{ ->}[r]^{a} & B \ar@{->}[r]^{b} & C  }$ with $ba \in I$ and $a$ has a cokernel, then $a$ is in $I$. 
\item[(S)] $I$ is closed under direct sums of morphisms. 
\end{itemize}
\end{definition}

A left weakly exact structure on $\A$ is given by dual axioms (Id$^{op}$), (P$^{op}$), (Q$^{op}$)(S$^{op}$). 

\begin{remark}
Note that, contrary to exact structures (see \cite[Proposition 2.9]{Bu}) the properties (S) and (S)$^{op}$ above are not implied by the rest of the properties and we need to add them. 
\end{remark}

\begin{theorem} \label{thm:right-left}
Let $\A$ be an additive category. A left weakly exact structure $\D$ on $\A$ can be combined with a right weakly exact structure $\I$ to form a weakly exact structure $\W$ given by the short exact sequences $\xymatrix{ A \ar@{ ->}[r]^{i} & B \ar@{->}[r]^{d} & C  }$ with $i \in I$ and $d \in \D$.
\end{theorem}
\begin{proof} 
We adapt the proof of \cite[Theorem 1]{Ru11} to the case of weakly exact structures.
\end{proof}

\begin{proposition} \label{prop:epi-left}
Every weakly exact structure $\W$ on $\A$ can be constructed from a right weakly exact structure and a left weakly exact structure as in  Theorem \ref{thm:right-left}. 
\end{proposition}\label{prop:left-right-weakly-exact}
\begin{proof} Let $\W$ be a weakly exact structure on an additive category $\A$. Let $I$ be the class of admissible monics of $\W$ and $\D$ the class of admissible epics of $\W$.
First, it is not difficult to show that $I$ and $\D$ are closed under isomorphisms. Second, since $\W$ satisfies $(E0)$ and $(E0)^{op}$, it is clear that $I$ satisfies (Id) and $\D$ satisfies (Id)$^{op}$. Third, by Proposition \ref{prop:obscure-axiom}, $I$ satisfies (Q) and $\D$ satisfies (Q)$^{op}$. And finally, since $\W$ is closed under direct sums it is clear that $I$ satisfies (S) and $\D$ satisfies (S)$^{op}$. 
\end{proof}

\subsection{The maximal weakly exact structure}\label{section:maximal weakly exact}

\begin{definition}\cite{RW77} A kernel $(A,f)$ in an additive category $\A$ is called \emph{semi-stable} if for every push-out square 
\[\xymatrix{
A \; \ar[d]_{t} 
\ar@{ ->}[r]^{f}  \ar@{}[dr]|{\text{PO}} 
& B\ar[d]^{s_B}\\
C \; \ar@{->}[r]_{s_C} & S}
\]

\noindent the morphism $s_C$ is also a kernel.
We define dually a \emph{semi-stable} cokernel.
A short exact sequence $\xymatrix{
A \, \ar@{ >->}[r]^{i} & B \ar@{ ->>}[r]^{d} & C
}$ in $\A$ is said to be \emph{stable} if
$i$ is a semi-stable kernel and $d$ is a semi-stable cokernel. We denote by $\E_{sta}$ the class of all \emph{stable} short exact sequences.
\end{definition}

We recall Crivei's characterisation of the maximum exact structures using the idempotent (or Karoubian) completion $H: \A \to \hat{\A}$ of the category $\A$:

\begin{theorem}\cite[Theorem 3.4]{Cr12} \label{thm: max_stable}
Let $\A$ be an additive category, and let $H: \A \to \hat{\A}$ be its idempotent completion. 
Then the class $\E_{sta}$ of stable short exact sequences of $\A$ defines an exact structure on $\A$ if and only if $\A$ is closed under pushouts and pullbacks for $(\hat{\A}, \hat{\E}_{max})$. In this case, $\E_{sta}$ is the maximal exact structure on $\A$.
\end{theorem}

Again, since the class of stable short exact sequences clearly forms the maximal class satisfying (E2) and (E2)$^{op}$, we get:

\begin{theorem}\label{cor:stable form maximum weakly exact}
Assume that $\A$ is closed under pushouts and pullbacks for $(\hat{\A}, \hat{\E}_{max})$ (for instance,  $\A$ is weakly idempotent complete).
Then the class of stable short exact sequences forms the unique maximal weakly exact structure on $\A$: 
\[\W_{max}=\E_{max}=\E_{sta} \]
\end{theorem}
\begin{proof}
By applying the arguments of the proof \cite[Theorem 3.4]{Cr12}.
\end{proof}



\section{Sub-bifunctors and closed sub-bifunctors of $\Ext^1$}\label{section:subfunctors}
\subsection{From weakly exact structures to bifunctors} \label{section:weakly to bifunctors}
Let $\W$ be a weakly exact structure on $\A$. The aim of this section is to  associate with $\W$ an additive functor to the category of abelian groups
\[\bW=\Ext^1_\W(-,-):\A^{op}\times \A \to \Ab.\]
This follows the classical construction for abelian categories:

\begin{definition}\label{definition:bifunctor}
Define for objects $A,C \in\A$ the set
\[ \bW(C,A) = \Ext^1_\W(C,A) = \left\{ \; \overline{(i,d)} \quad | \quad \xymatrix{ A \ar@{ ->}[r]^{i} & B \ar@{->}[r]^{d} & C  } \; \in \W\right\},
\]

where we denote by $\overline{(i,d)}$ the usual equivalence class of the short exact sequence $(i,d)$.
The action of the functor $\bW$ on morphisms is defined by taking the pushout or pullback along morphisms.
Moreover, one can define on $\bW(C,A)$ an addition by Baer's sum.

Given a left weakly exact structure $\D$ on $\A$ and objects $A,C \in\A$, we define
\[ \bD_A(C) =  \left\{ \; \overline{(i,d)} \quad | \quad \xymatrix{ A \ar@{ ->}[r]^{i} & B \ar@{->}[r]^{d} & C  } \mbox{ is a short exact sequence with }  d\in \D\right\}
\]
We also use the notation $ \Ext_\D^1(C,A) = \bD_A(C)$. Dually, we define $ \Ext_\I^1(C,A) = \bI_A(C)$ for a right weakly exact structure $\I$.
\end{definition}
The following results are easily adapted from the abelian context:
\begin{lemma}\label{lemma:right-functor}
Let $\D$ be a left weakly exact structure on $\A$.
Then for each $A \in \A$, the construction in Definition \ref{definition:bifunctor} yields a  functor $\bD_A=\Ext_\D^1(-,A):\A^{op} \to \Set$.
Dually, for every object $C \in \A$, a right weakly exact structure $\I$ defines a  functor $\bI_C=\Ext_\I^1(C,-):\A \to \Set$.
\end{lemma}

\begin{proposition}\label{proposition:bifunctor} Let $\W$ be weakly exact structure on $\A$.
Then the construction in Definition \ref{definition:bifunctor} yields an additive bifunctor \[\bW=\Ext^1_{\W}(-,-):\A^{op} \times \A \longrightarrow \Ab ; \; (C, A)\longmapsto \bW(C,A).\]
\end{proposition}

\begin{lemma}\label{lemma:sub-bifunctor}
Let $\V$ and $\W$ be weakly exact structures on $\A$ with $\V \subseteq\W$.
Then $\bV=\Ext_\V^1(-,-):\A^{op}\times \A \to \Ab$ is an additive sub-bifunctor of $\bW$.
\end{lemma}

\begin{remark} \label{remark:Ext1}
We consider the partial order on $\Bf(\A)$ given by
\[F\leq F' \Longleftrightarrow F(C,A) \leq F'(C,A) \mbox{ for all } A,C \in \A\]
that is, $F(C,A)$ is a subgroup of $F'(C,A)$ for every pair of objects in $\A$. 
The construction in Definition \ref{definition:bifunctor} thus defines a map $\Phi$ from the weakly exact structures included in $\E_{max}$ on the additive category $\A$ to the $\A-\A-$bimodules:
\begin{align*}
\Phi : \Wex(\A) & \longrightarrow \Bf(\A)\\ \W & \longmapsto \bW=\Ext^1_{\W}(-,-).
\end{align*}
\end{remark}

Lemma \ref{lemma:sub-bifunctor} shows that $\Phi$ is a morphism of posets. 
The elements in $\Ex(\A)$ are sent under the map $\Phi$ to subfunctors of $\Ext^{1}_{\A}(-,-)=\bE_{max}$  that enjoy an additional property, namely they give rise to a long exact sequence of functors:

\begin{definition}(\cite{BH,DRSS}) \label{definition:closed}
An additive sub-bifunctor $F$ of $\Ext^{1}_{\A}(-,-)$ is called {\em closed} if for any short exact sequence 
\[\xymatrix{E: \,  A \ar@{ ->}[r]^{i} & B \ar@{->}[r]^{d} & C }\] whose class lies in $F(C,A)$ and any object $X$ in $\A$, the sequences
\[ 0 \to \Hom(X,A) \to \Hom(X,B) \to \Hom(X,C) \to F(X,A) \to F(X,B) \to F(X,C)
\]
and 
\[ 0 \to \Hom(C,X) \to \Hom(B,X) \to \Hom(A,X) \to F(C,X) \to F(B,X) \to F(A,X)
\]
are exact in the category of abelian groups. 
\end{definition}
\begin{proposition}\cite[Prop 1.4]{DRSS}
Let $\E$ be an exact structure on $\A$. Then the bifunctor $\Phi(\E)$ is closed.
\end{proposition}

\subsection{From sub-bifunctors of $\Ext^1_\A$ to weakly exact structures}\label{section:sub-bifunctors to weakly}

We construct a partial inverse of the function $\Phi$. 
Consider the interval 
\[\Wex(\E_{max}):=[\E_{min}, \E_{max}] \subseteq \Wex(\A).\] 
Likewise, we write $\Bf(\bE_{max})$ for the class of sub-objects of $\bE_{max}$ in $\Bf(\A)$:
\[\Bf(\bE_{max}) := [\bE_{min},\bE_{max}] \subseteq \Bf(\A)\]
is the interval of all additive bifunctors between the minimum and the maximum exact structure on $\A.$
Note that for a weakly idempotent complete category $\A$, we have that $\E_{max}$ is the maximal weakly exact structure on $\A$, therefore $\Wex(\E_{max}) = \Wex(\A).$


\begin{definition}\cite{BH,DRSS}\label{def:Psi}
Let $F: \A^{op} \times \A \longrightarrow \Ab$ be an additive sub-bifunctor of $\Ext^{1}_{\A}(-,-)$.
Define  $\W_{F}$ to be the class of short exact sequences formed by all classes $\overline{(i,d)} \in F(C,A).$
\end{definition}
\begin{proposition}(\cite{BH,DRSS})\label{prop:bijections}
The construction in Definition \ref{def:Psi} yields a map 
\begin{align*}
\Psi :  \Bf(\bE_{max}) & \longrightarrow \Wex(\E_{max})\\
F & \longmapsto \W_{F}.
\end{align*}
Moreover, the functions $\Phi$ and $\Psi$ induce mutually inverse poset isomorphisms
\[\begin{array}{ccc}
 \Wex(\E_{max}) & \longleftrightarrow & \Bf(\bE_{max}) \\
 \cup &   & \cup\\
\Ex(\A) & \longleftrightarrow & \Cbf(\A)
\end{array}\]
where $\Cbf(\A)$ denotes the subclass of closed sub-bifunctors of $\Ext_\A^{1}$.
\end{proposition}

\subsection{Example}\label{section:example}

We reconsider here Example \ref{A3} in light of the bijection from the last proposition: Let $\A = \rep Q$ be the category of representations of  the quiver
\[ Q : \qquad \xymatrix{ 1 \ar[r]  & 2 \ar[r] & 3 }\]

The Auslander-Reiten quiver of $\A$ is as follows:
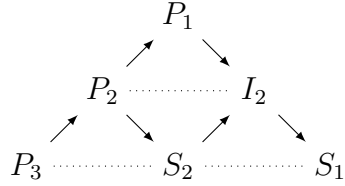
\begin{center}
\begin{tikzpicture}

\node (a) at (-1,0) {$P_2$};
\node (b) at (0,1) {$P_1$};
\node (c) at (0,-1) {$S_2$};
\node (d) at (1,0) {$I_2$};
\node (e) at (-2,-1) {$P_3$};
\node (f) at (2,-1) {$S_1$};
 \draw [-latex] (a) -- (b);
 \draw [-latex] (a) -- (c);
 \draw [-latex] (b) -- (d);
 \draw [-latex] (c) -- (d);
 \draw [-latex] (e) -- (a);
 \draw [-latex] (d) -- (f);
 \draw [dotted,-] (e) -- (c);
 \draw [dotted,-] (c) -- (f);
 \draw [dotted,-] (a) -- (d);
 \end{tikzpicture}
 \end{center}

There are (up to equivalence) exactly five non-split exact sequences with indecomposable end terms, where the first three are the Auslander-Reiten sequences:

\begin{enumerate}
\item[($\alpha$)] \hspace{2.5cm}$\xymatrix{ 0 \ar[r] & P_3 \ar[r]^{a} & P_2 \ar[r]^c & S_2  \ar[r] & 0}$
\item[($\beta$)] \hspace{2.5cm}$ \xymatrix{ 0 \ar[r] & S_2 \ar[r]^e & I_2 \ar[r]^f & S_1 \ar[r]  & 0 }$
\item[($\gamma$)] \hspace{2cm}$ \xymatrix{ 0 \ar[r] & P_2 \ar[r] & P_1\oplus S_2 \ar[r] & I_2 \ar[r]  & 0 }$
\item[($\delta$)] \hspace{2.5cm}$ \xymatrix{ 0 \ar[r] & P_3 \ar[r]& P_1 \ar[r]^d & I_2  \ar[r] & 0}$
\item[($\epsilon$)] \hspace{2.5cm}$ \xymatrix{ 0 \ar[r] & P_2 \ar[r]^b & P_1 \ar[r] & S_1 \ar[r] & 0}$
\end{enumerate}

An additive functor is uniquely determined by its values on indecomposable objects. To study additive sub-bifunctors of $\Ext^1_{\A}$ it is therefore sufficient to examine the bifunctor structure on the vector space generated by these five non-split exact sequences with indecomposable end-terms. It is depicted in the following diagram, which indicates the multiplication rules 
 $\delta e=\alpha, a\delta=\gamma, \epsilon f=\gamma, c\epsilon=\beta$ (see Definition 3.1) :
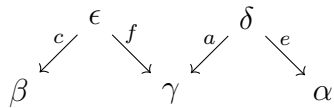
\begin{center}
\begin{tikzpicture}

\node (a) at (-1,0) {$\epsilon$};
\node (b) at (-2,-1) {$\beta$};
\node (c) at (0,-1) {$\gamma$};
\node (d) at (1,0) {$\delta$};
\node (f) at (2,-1) {$\alpha$};
   \path[->] 
  (d)  edge node[above]  {$\scriptstyle e$} (f)
  (d)  edge node[above]  {$\scriptstyle a$} (c)
  (a)  edge node[above]  {$\scriptstyle f$} (c)
  (a)  edge node[above]  {$\scriptstyle c$} (b)
  ;       
 \end{tikzpicture}
 \end{center}

It is easy to see that $\Ext^1_{\A}$ admits 13 additive sub-bifunctors, 
and the sub-bifunctor lattice is given in Figure 2, 
indicating each sub-bifunctor by a set of generators.
The eight closed sub-bifunctors, corresponding to the eight exact structures on $\A$, are indicated in {\color{blue}blue}.
Note that the sub-bifunctor generated by the set of all Auslander–Reiten sequences $\{\alpha, \beta, \gamma\}$ corresponds to the Auslander–Reiten phantom morphisms studied in \cite{FGHT}.

\bigskip

\begin{center}
   
   \begin{tikzpicture}
\node (0) at (0,-2) {\color{blue}$\emptyset$};
\node (1) at (-2,-1) {\color{blue}$\alpha$};
\node (2) at (0,-1) {\color{blue}$\gamma$};
\node (3) at (2,-1) {\color{blue}$\beta$};
\node (4) at (-2,0) {$\alpha,\gamma$};
\node (5) at (-0,0) {\color{blue}$\alpha,\beta$};
\node (6) at (2,0) {$\beta,\gamma$};
\node (7) at (-2,1) {\color{blue}$\alpha,\gamma,\delta$};
\node (8) at (0,1) {$\alpha,\beta,\gamma$};
\node (9) at (2,1) {\color{blue}$\beta,\gamma,\epsilon$};
\node (10) at (-1,2) {$\alpha,\beta,\gamma,\delta$};
\node (11) at (1,2) {$\alpha,\beta,\gamma,\epsilon$};
\node (12) at (0,3) {\color{blue}$\alpha,\beta,\gamma,\delta,\epsilon$};
 \draw [-] (0) -- (1);
 \draw [-] (0) -- (2);
 \draw [-] (0) -- (3);
 \draw [-] (1) -- (4);
 \draw [-] (1) -- (5);
 \draw [-] (2) -- (4);
 \draw [-] (2) -- (6);
\draw [-] (3) -- (6);
 \draw [-] (3) -- (5);
\draw [-] (4) -- (7);
 \draw [-] (4) -- (8);
\draw [-] (5) -- (8);
 \draw [-] (6) -- (8);
 \draw [-] (6) -- (9);
 \draw [-] (7) -- (10);
 \draw [-] (8) -- (10);
 \draw [-] (8) -- (11);
 \draw [-] (9) -- (11);
 \draw [-] (10) -- (12);
 \draw [-] (11) -- (12);
    \end{tikzpicture}
    
  Figure 2:  sub-bifunctors of $\Ext^1_\A(-,-)$
    \label{figure:subbimodules}
\end{center}

\subsection{Weakly exact structures as bimodules}\label{sec:bimodules}

\begin{definition}
Let $\mathcal{A}$ be an additively finite, Hom-finite Krull-Schmidt category with indecomposables $X_1, \ldots , X_n$ and denote by $A = \End(X)$ with $X=X_1 \oplus \cdots \oplus X_n$ its Auslander algebra.
The Krull-Schmidt property implies that the additive category $\A$ is weakly idempotent complete, thus as discussed in Section \ref{section:maximal weakly exact}, we know that the maximum weakly exact structure coincides with the maximum weakly exact structure  formed by the stable short exact sequences. 
The corresponding bifunctor $\bE_{max}$, evaluated at the object $X$ yields a bimodule 
\[ \B= \bE_{max}(X,X)\]
over the Auslander algebra $A$.


We denote by $\Bim(\B)$ the class of all sub-bimodules of $_A\B_A$; it forms a poset ($\Bim(\B)$,$\subseteq$) with inclusion as order relation.
\end{definition} 

\begin{example}
In the example studied in Section \ref{section:example}, the Auslander algebra $A$ is the algebra whose quiver is the Auslander-Reiten quiver with mesh relations, and the $A-A-$bimodule $\B=\bE_{max}(X,X)$ is the $\Ext-$bimodule on $A$, a five-dimensional bimodule with basis given by the elements $\alpha,\beta,\gamma,\delta,
\epsilon$. 
The Figure 2 
describes the bimodule lattice ($\Bim(\B)$,$\subseteq$) in this example.
\end{example}

\section{Weakly extriangulated structures}
 Extriangulated structures \cite{NakaokaPalu1} (or, equivalently, $1-$exangulated structures \cite{HLN}) generalize both exact and triangulated categories. In this Section, we generalise these categories, by defining their weak versions.
 
\begin{definition}\cite{HLN}
Let $\mathbb{E}: \ca^{op} \times \ca \to \Ab$ be an additive bifunctor. Given a pair of objects $A, C \in \ca,$ we call an element $\delta \in \mathbb{E}(C, A)$ an $\mathbb{E}-$extension. When we want to emphasize $A$ and $C$, we also write ${}_{A}\delta_{C}$.
\end{definition}

Since $\mathbb{E}$ is a bifunctor, each morphism $a \in \Hom(A, A')$ induces the extension $a_{\ast}(\delta) := \mathbb{E}(C, a) (\delta) \in \mathbb{E}(C, A')$. 
Similarly, we set $c^{\ast}(\delta) := \mathbb{E}(c, A) (\delta) \in \mathbb{E}(C', A)$.

Moreover, we have $\mathbb{E}(c, a) (\delta) = c^{\ast} a_{\ast} (\delta) = a_{\ast} c^{\ast} (\delta).$

By the Yoneda lemma, each extension ${}_{A}\delta_{C}$ induces a pair of natural transformations
$\delta_{\sharp}: \Hom(-, C) \to \mathbb{E}(-, A) \ \ \mbox{and} \ \ \delta^{\sharp}: \Hom(A, -) \to \mathbb{E}(C, -)$ by
\begin{align*}
(\delta_{\sharp})_X: \Hom(X, C) \to \mathbb{E}(X, A), \ \ c \mapsto c^{\ast}(\delta);\\
(\delta^{\sharp})_X: \Hom(A, X) \to \mathbb{E}(C, X), \ \ a \mapsto a_{\ast}(\delta).
\end{align*}

\begin{definition} \label{morphism:extensions}
A morphism of extensions ${}_{A}\delta_{C} \to {}_{B}\rho_{D}$ is a pair of morphisms $(a, c) \in \Hom(A, B) \times \Hom (C, D)$ such that $a_{\ast}(\delta) = c^{\ast}(\rho).$
\end{definition}

\begin{definition}
A \emph{weak cokernel} of a morphism $f: A \to B$ in $\ca$ is a morphism
$g : B \to C$ such that 
the sequence of functors 
$$\Hom(C, -) \to \Hom(B, -) \to \Hom(A, -)$$
is exact.
Equivalently, $g$ is a weak cokernel of $f$ if $g \circ f = 0$ and for each morphism
$h: B \to X$ such that $h \circ f = 0,$ there exists a (not necessarily unique) morphism
$l: C \to X$
such that $h = l \circ g.$
Weak kernel is a weak cokernel in $\ca^{op}.$
\end{definition}

Note that weak (co)kernels satisfy the same factorization properties as usual (co)kernels, but without requiring uniqueness.
Clearly, a weak (co)kernel $g$ of $f$ is a (co)kernel of $f$ if and only if $g$ is a monomorphsim (resp. an epimorphism).

\begin{definition}
We call a pair of composable morphisms 
$A \overset{f}\to B \overset{g}\to C$
a \emph{weak kernel-cokernel pair} if $f$ is a weak kernel of $g$ and $g$ is a weak cokernel of $f$.
\end{definition}

By definition, in each weak kernel-cokernel pair as above the composition $g \circ f$ is $0,$ so the pair can be understood as an element of the category $\cc^{[0, 2]} (\ca) \hookrightarrow \cc(\ca)$ of complexes over $\ca$ concentrated in the degrees $0, 1$ and $2.$


Let $\cc_{w}(\ca)$ be the full subcategory of  $\cc^{[0, 2]} (\ca)$ with objects being weak kernel-cokernel pairs.

Consider morphisms of complexes in  $\cc_{w}(\ca)$
\begin{align} \label{morphism:wkc}
\xymatrix@R=0.6cm{
*+++{A} \ar@{=}[d]_{1_A} \ar@{->}[r]^{f} &*+++{B} \ar@{->}[d]_{b} \ar@{->}[r]^{g}  &*+++{C} \ar@{=}[d]_{1_C} \\
*+++{A}  \ar@{->}[r]^{f'} &*+++{B'} \ar@{->}[r]^{g'}  &*+++{C}
}
\end{align}
with leftmost and rightmost vertical morphisms being identities. 

\begin{lemma} \label{lemma:equivalence:wkc}
For a diagram of the form (\ref{morphism:wkc}), the following are equivalent:
\begin{itemize}
\item The morphism $h^{\bullet} = (1_A, b, 1_C)$ is an isomorphism in $\cc_{w}(\ca);$
\item The morphism $b$ is an isomorphism;
\item The morphism $h^{\bullet}$ is a homotopy equivalence in $\cc^{[0, 2]}(\ca)$. 
\end{itemize}
\end{lemma} 

Here by homotopy equivalence in $\cc^{[0, 2]}(\ca)$ we mean that there exists a morphism $k^{\bullet}$ in $\cc^{[0, 2]}(\ca)$ and morphisms 
$$\phi_1: B \to A, \  \phi_2: C \to B, \ \psi_1: B' \to A, \ \psi_2: C \to B'$$ 
such that the pair $(\phi_1, \phi_2)$ yields a chain homotopy $k^{\bullet} \circ h^{\bullet} \sim 1$ and the pair $(\psi_1, \psi_2)$ yields a chain homotopy $h^{\bullet} \circ k^{\bullet} \sim 1.$

\begin{proof}
This is a reformulation of \cite[Lemma 4.1]{HLN}, see also \cite[Claim 2.8]{HLN}.
\end{proof}

Morphisms $h^{\bullet} = (1_A, b, 1_C)$ satisfying either of conditions in Lemma \ref{lemma:equivalence:wkc} define an equivalence relation on objects in $\cc_{w}(\ca).$
We denote by $[A \overset{f}\to B \overset{g}\to C]$ the equivalence class of the complex $A \overset{f}\to B \overset{g}\to C$ in $\cc_w(\ca)$ under this equivalence.

\begin{definition} (cf. \cite[Definition 2.22]{HLN})
Let $\mathfrak{s}$ be a correspondence which associates an equivalence class $\mathfrak{s}(\delta) = [A \overset{f}\to B \overset{g}\to C]$ 
in $\cc_{(\ca)}$ to each extension $\delta = {}_A\delta_C$. Such $\mathfrak{s}$ is called a \emph{realization}
of $\mathbb{E}$ if it satisfies the following condition for any $\mathfrak{s}(\delta) = [A \overset{f}\to B \overset{g}\to C]$
and any $\mathfrak{s}(\rho) = [A' \overset{f'}\to B' \overset{g'}\to C']:$
\begin{itemize}
\item[(R0)] 
For any morphism of extensions $(a, c): \delta \to \rho,$ there exists a morphism $b: B \to B'$ such that
$h^{\bullet} = (a, b, c)$ is a morphism in $\cc^{[0, 2]}(\ca):$

\begin{align*} 
\xymatrix@R=0.3cm{
*+++{A} \ar@{=}[dd]_{1_A} \ar@{->}[rr]^{f} & &*+++{B} \ar@{->}[dd]_{b} \ar@{->}[rr]^{g}  & &*+++{C} \ar@{=}[dd]_{1_C} \\
& \circlearrowright & & \circlearrowright & \\
*+++{A}  \ar@{->}[rr]^{f'} & &*+++{B'} \ar@{->}[rr]^{g'}  & &*+++{C.}
}
\end{align*}
Such $h^{\bullet}$
is called a
\emph{lift} of $(a, c).$
\end{itemize}

We say that $[A \overset{f}\to B \overset{g}\to C]$
\emph{realizes} $\delta$ whenever we have $\mathfrak{s}(\delta) = [A \overset{f}\to B \overset{g}\to C].$
\end{definition}

Each weak kernel-cokernel pair $A \overset{f}\to B \overset{g}\to C$ realizing an extension $\delta$  induces a pair of sequences of functors
\begin{align} \label{transformations}
\Hom(C, -) \to \Hom(B, -) \to \Hom(A, -) \to \mathbb{E}(C, -);\\
\Hom(-, A) \to \Hom(-, B) \to \Hom(-, C) \to \mathbb{E}(-, A).
\end{align}

\begin{definition}(cf. \cite[Definition 2.22]{HLN})

A realization $\mathfrak{s}$ is called \emph{exact} if the following two conditions are satisfied:
\begin{itemize}
    \item[(R1)] For each extension $\delta$, for each $A \overset{f}\to B \overset{g}\to C$ realizing $\delta$, both sequences (\ref{transformations}) are exact (i.e. exact when applied to each object in $\ca$);
    \item[(R2)] For each object $A \in \ca,$ we have
    $$\mathfrak{s}({}_A{0}_0) = [A \overset{1_A}\to A \to 0], \ \ \mathfrak{s}({}_0{0}_A) = [0 \to A \overset{1_A}\to A].$$
\end{itemize}
\end{definition}

\begin{remark}
Note that since we require realizations to be given by weak kernel-cokernel pairs, sequences (\ref{transformations}) are automatically exact at $\Hom(B, -),$ resp. at $\Hom(-, B).$ In other words, condition (R1) concerns only exactness at $\Hom(A, -),$ resp. at $\Hom(-, C).$ 
\end{remark}

\begin{remark}
By \cite[Proposition 2.16]{HLN}, condition (R1) does not depend on the choice of a representative in the equivalence class $\mathfrak{s}(\delta).$
\end{remark}

\begin{definition} (\cite[Definition 2.23]{HLN}, \cite[Definition 2.15, Definition 2.19]{NakaokaPalu1})
Let $\mathfrak{s}$ be an exact realization of $\mathbb{E}.$ Pairs $\delta, \mathfrak{s}(\delta)$ are called (\emph{distinguished}) $\mathbb{E}-$triangles. If a complex 
$$A \overset{f}\to B \overset{g}\to C$$
is a representative in $\mathfrak{s}(\delta)$ for some $\delta,$ it is called a \emph{conflation}. In this case, the morphism $f$ is called an \emph{inflation} and the morphism $g$ is called a \emph{deflation}.
\end{definition}

\begin{lemma}(\cite[Proposition 3.2]{HLN})
The class of conflations and the class of $\mathbb{E}-$triangles are both closed under direct sums and direct summands.
\end{lemma}


\begin{definition} \cite[Definition 2.27]{HLN}
Let $f^{\bullet} = (1_A, b, c)$ be a morphism in $\cc^{[0, 1]}(\ca):$
\begin{align*}
\xymatrix@R=0.6cm{
*+++{A} \ar@{=}[d]_{1_A} \ar@{->}[r]^{f} &*+++{B} \ar@{->}[d]_{b} \ar@{->}[r]^{g}  &*+++{C} \ar@{->}[d]_{c} \\
*+++{A}  \ar@{->}[r]^{f'} &*+++{B'} \ar@{->}[r]^{g'}  &*+++{C'.}
}
\end{align*}
Its \emph{mapping cone} $M_{f}^{\bullet}$ is defined to be the complex
\begin{align*}
B \overset{\begin{bmatrix} -g \\ b \end{bmatrix}}\to C \oplus B' \overset{\begin{bmatrix} c \ g' \end{bmatrix}}\to C'.
\end{align*}
In other words, this is the usual mapping cone of the morphism of complexes
\begin{align*}
\xymatrix@R=0.6cm{
*+++{B} \ar@{->}[d]_{b} \ar@{->}[r]^{g}  &*+++{C} \ar@{->}[d]_{c} \\
*+++{B'} \ar@{->}[r]^{g'}  &*+++{C'.}
}
\end{align*}
The \emph{mapping cocones} (cylinder) of morphisms of the form $(a, b, 1_C)$ are defined dually.
\end{definition}

\begin{definition} (\cite[Definition 2.32 for $n = 1$]{HLN})
A 1-exangulated category is a triplet $(\ca, \mathbb{E}, \mathfrak{s})$ of an additive category $\ca$, additive bifunctor $\mathbb{E}: \ca^{op} \times \ca \to \Ab$, and its exact realization $\mathfrak{s}$, satisfying
the following conditions.
\begin{itemize}
    \item[(EA1)] The composition of two inflations is an inflation. Dually, the composition of two deflations is a deflation.
    \item[(EA2)]  For each $\rho \in \mathbb{E}(C', A)$ and $c \in \Hom(C, C'),$ for each pair of realizations 
    $A \overset{f}\to B \overset{g}\to C$ of $c^{\ast} \rho$ and $A \overset{f'}\to B' \overset{g'}\to C'$ of $\rho,$ the morphism $(1_A, c): c^{\ast} \rho \to \rho$ admits a \emph{good lift} $f^{\bullet} = (1_A, b, c)$, in the sense that $M_{f}^{\bullet}$ realizes $f_{\ast} \rho.$
   \item[(EA2)$^{op}$] Dual of (EA2).
\end{itemize}
\end{definition}

\begin{proposition}(\cite[Proposition 4.3]{HLN})
A triplet $(\ca, \mathbb{E}, \mathfrak{s})$ is a 1-exangulated category if and only if it is an extriangulated category as defined by Nakaoka and Palu \cite{NakaokaPalu1}.
\end{proposition}

This result motivates the following definition.

\begin{definition}\label{weakly $1-$exangulated structures}
A \emph{weakly extriangulated} (= \emph{weakly $1-$exangulated}) structure on an additive category $\ca$ is  a pair $(\mathbb{W}, \mathfrak{s})$ of an additive bifunctor $\mathbb{W}: \ca^{op} \times \ca \to \Ab$ and its exact realization $\mathfrak{s}$ satisfying axioms (EA2) and (EA2)$^{op}$. 

\end{definition}

Let $(\ca, \mathbb{W}, \mathfrak{s})$ be a weakly extriangulated structure. Assume that $\mathbb{W'}$ is an additive sub-bifunctor of $\mathbb{W}.$ Consider the restriction $\mathfrak{s}|_{\mathbb{W}'}$ of the realization $\mathfrak{s}$ on $\coprod\limits_{c,a \in \ca} \mathbb{W}'(c, a).$ The following immediately follows from the definitions. The case of $(\mathbb{W}, \mathfrak{s})$ extriangulated was considered in \cite[Claim 3.8]{HLN}.

\begin{lemma} \label{wex_substr}
$(\mathbb{W'}, \mathfrak{s}|_{\mathbb{W}'})$ is a weakly extriangulated structure on $\ca.$
\end{lemma}

We say that $(\mathbb{W'}, \mathfrak{s}|_{\mathbb{W}'})$ is a \emph{weakly extriangulated substructure} of  $(\mathbb{W}, \mathfrak{s})$.

\begin{lemma}
A weakly exact structure $\cW$ on $\ca$ defines a weakly extriangulated structure $(\ca, \mathbb{W}, \mathfrak{s}).$
\end{lemma}

\begin{proof}
Using Lemma \ref{factorization},
all the arguments from \cite[Example 2.13]{NakaokaPalu1}, except for those concerning (ET4) and (ET4)$^{\text{op}}$, apply here word for word. That means that a weakly exact structure defines a pair of a bifunctor and its exact realization. Axioms (EA2) and (EA2)$^{op}$ follow directly from axioms (E2) and  (E2)$^{\text{op}}$ combined with Lemma \ref{prop-2.12} and its dual.
\end{proof}


We can also characterize weakly exact structures among weakly extriangulated ones.

\begin{lemma} (cf. \cite[Corollary 3.18]{NakaokaPalu1})
Let $(\ca, \mathbb{W}, \mathfrak{s})$ be a weakly extriangulated category, in which each inflation
is monomorphic, and each deflation is epimorphic. If we let $\cW$ be the class of
conflations given by the $\mathbb{W}-$triangles, then $(\ca, \cW)$ is a weakly exact
category.
\end{lemma}

\begin{proof}
If an inflation in a conflation is monomorphic, it is not just a weak kernel of the deflation, but \emph{the} actual kernel. Similarly, if a deflation is epimorphic, it is the cokernel of an inflation. Therefore, if each inflation
is monomorphic, and each deflation is epimorphic, all conflations are kernel-cokernel pairs. From the exactness of the realization, it follows that the class of conflations is closed under direct sums and axioms (E0) and (E0)$^{op}$ are satisfied. Axioms (EA2) and (EA2)$^{op}$ imply the axioms (E2) and (E2)$^{op}$ by Lemma \ref{prop-2.12} and its dual.
\end{proof}

Breaz and Modoi \cite{BM} introduced the notions of \emph{almost exact structures} on full extension-closed subcategories $\mathcal{A}$ of triangulated categories $\mathcal{T}$ in terms of \emph{proper classes of triangles} (generalizing work of Beligiannis \cite{Bel}) and of \emph{phantom $\mathcal{A}$ ideals} of morphisms in $\mathcal{T}$. They found \cite[Proposition 2.2.4]{BM} a bijection between almost exact structures on $\mathcal{A}$ and phantom $\mathcal{A}-$ideals. 


\begin{lemma}
Each pair of a phantom $\ca-$ideal in $\mathcal{T}$ and the corresponding proper class of triangles yields a weakly extriangulated structure on $\mathcal{A}.$
\end{lemma}

\begin{proof}
This follows from \cite[Remark 2.2.3 (ii), Proposition 2.2.4]{BM}, the fact that $\ca$ is extriangulated with the structure induced by that of $\ct$, and Lemma \ref{wex_substr}.
\end{proof}


\section{Defects and topologizing subcategories}

We extend the notion of contravariant defects to the setting of weakly extriangulated categories. These categories were used in \cite{Buan, En18, Enomoto3, FG} to classify exact structures on an additive category and, more generally, extriangulated substructures of an extriangulated structure. We show that their results extend to our framework.

\begin{definition}
Let $\ca$ be an essentially small additive category. Contravariant additive functors $\cA^{op} \to \Ab$ to the category of abelian groups are called \emph{right $\ca-$modules}. They form an abelian category $\Mod \ca.$ Dually, \emph{left $\ca-$modules} are covariant additive functors to abelian groups, they form an abelian category that can be seen as $\Mod \ca^{op}.$ 
\end{definition}
These categories have enough projectives. Those are precisely the direct summands of direct sums of representable functors $\Hom(-, A) \in \Mod \ca,$ resp. $\Hom(A, -) \in \Mod \ca^{op}.$ 

\begin{definition}
An $\cA$-module $M$ is called 
\emph{coherent} if it is finitely presented and each of its finitely generated submodule is also finitely presented. Note that every finitely generated submodule of a coherent module is automatically coherent.
\end{definition}

By definition, we have a chain of embeddings of full additive categories
$$\coh(\ca) \hookrightarrow \fp(\ca) \hookrightarrow \fg(\ca) \hookrightarrow \Mod \ca,$$

where the first three categories are the categories of coherent, finitely presented and finitely generated right $\ca-$modules, respectively.

The category of finitely presented modules $\fp(\ca)$ is known to be abelian if and only if the category $\ca$ has weak kernels. The category of coherent modules behaves better, as the following standard fact shows:

\begin{proposition} (\cite[Proposition 1.5]{Herzog97}, see also \cite[Appendix B]{Fiorot16})
The category $\coh(\ca)$ is abelian and the canonical embedding $\coh(\ca) \hookrightarrow \Mod \ca$ is exact. In particular, $\coh(\ca)$ is closed under kernels, cokernels and extensions in $\Mod(\ca)$
\end{proposition}

Two more important full subcategories of categories of modules over abelian categories have been studied thoroughly since 1950s and 1960s: the category of effaceable functors, 
and the category of defects introduced by Auslander. 
These notions have been generalized to the  setting of exact categories (see e.g. \cite{Ke, Fiorot16, En18}) and, by Ogawa \cite{Ogawa} and Enomoto \cite{Enomoto3}, to that of extriangulated categories. Ogawa's definition uses only part of the axioms of extriangulated categories, and so we can formulate it in our broader context.

Let $(\cA, \mathbb{W}, \mathfrak{s})$ be a weakly extriangulated category.

\begin{definition}
We say that a module $F \in \Mod \ca$ is \emph{weakly effaceable with respect to $(\mathbb{W}, \mathfrak{s})$} if the following condition is satisfied:

For any $Z \in \ca$ and any $z \in F(Z)$, there exists a deflation $g: Y \twoheadrightarrow Z$ such that $F(g)(z) = 0$.
\end{definition}

\begin{definition}
Given a conflation $X \overset{f}\rightarrowtail Y \overset{g}\twoheadrightarrow Z,$ we define its \emph{contravariant defect} to be the cokernel of $\Hom(-, g): \Hom(-, Y) \to \Hom(-, Z)$ in $\Mod \ca$.
\end{definition}

We denote by $\Eff \mathbb{W}$ the category of weakly effaceable functors and by $\defff \mathbb{W}$ the full subcategory of right $\ca-$modules isomorphic to defects of conflations. If $(\cA, \mathbb{W}, \mathfrak{s})$ corresponds to a weakly exact structure $\cW$ on $\cA$, we also write $\Eff \cW := \Eff \mathbb{W}$ and $\defff \cW := \Eff \cW.$

For abelian categories endowed with maximal exact structures, the following two statements are standard.

\begin{lemma}
The category $\Eff \mathbb{W}$ is closed under subquotients and finite direct sums in $\Mod \ca$.
\end{lemma}

\begin{proof}
Let $0 \to F \overset{\mu}\to G \overset{\nu}\to H \to 0$
be a short exact sequence in $\Mod \ca.$ Assume that $G$ is weakly effaceable with respect to $(\mathbb{W}, \mathfrak{s})$. Let $Z$ be an object of $\ca$. Choose an element $z \in F(Z)$ and a deflation $f: P \to Z$ such that
$$0 = G(f) \circ \mu(Z)(z) = \mu(P) \circ F(f)(z).$$
Since $\mu$ is monic, $F(f)(z) = 0.$ Thus, $F$ is weakly effaceable with respect to $(\mathbb{W}, \mathfrak{s})$. So $\defff \mathbb{W}$ is closed under subobjects. The rest is proved by similar straightforward diagram chasing.
\end{proof}

\begin{lemma}
The category $\defff \mathbb{W}$ is closed under kernels and cokernels in $\Mod \ca$.
\end{lemma}

\begin{proof}
The same argument as in \cite[Lemma 2.6]{Ogawa} applies here. A morphism of defects of two conflations gives rise to a morphism $(a, c)$ of these conflations. Then the kernel is given by the defect of the mapping cone of any good lift of the morphism $(1, c)$ and the cokernel is given by the defect of the mapping cocone of any good lift of the morphism $(a, 1).$
\end{proof}

The following notion was introduced by Rosenberg \cite{Rosenberg} in his works on noncommutative algebraic geometry and reconstruction of schemes.

\begin{definition}
A full subcategory of an abelian category is called \emph{topologizing} if it is closed under subquotients and finite direct sums.
\end{definition}

\begin{proposition}
Let $(\cA, \mathbb{W}, \mathfrak{s})$ be a weakly extriangulated category.
 We have $$\defff \mathbb{W} = \Eff \mathbb{W} \bigcap \coh (\ca)$$ and this category is topologizing.
\end{proposition}

\begin{proof}
The same argument as in the proof of \cite[Proposition 2.9]{Enomoto3} applies here. The only difference is that in our generality $\Eff \mathbb{W}$ is not closed under extensions in $\Mod \ca,$ but only under finite direct sums. \end{proof}

For $\ca-$modules, we have natural notions of subobjects, quotients and extensions: these are defined object-wise (for objects in $\ca$).

\begin{definition} \label{def:topologizing}
We say that a subcategory  of an arbitrary  (not necessarily abelian) full subcategory $\cc$ of $\coh(\ca)$ is \emph{topologizing} if it is closed under subquotients (considered object-wise) and finite direct sums. Equivalently, it is topologizing if it is a full subcategory of $\cc$ which is topologizing in $\coh(\ca).$ 

Similarly, we say that a subcategory of $\cc$ is \emph{Serre} if it is topologizing and closed under extensions; equivalently, if it is a full subcategory of $\cc$ and a Serre subcategory in $\coh(\ca).$
Note that this definition ensures that a Serre subcategory of $\cc$ is abelian.
\end{definition}

\begin{cor}
Let $(\cA, \mathbb{W}, \mathfrak{s})$ be a weakly extriangulated category and let $(\cA, \mathbb{W'}, \mathfrak{s}|_{\mathbb{W}'})$ be a weakly extriangulated substructure on $\ca$. 
Then the category $\defff \mathbb{W}'$ is a topologizing subcategory of $\defff \mathbb{W}.$
\end{cor}

\begin{cor}
Let $\cW'$ be a weakly exact substructure of a weakly exact structure $\cW.$ Then the category $\defff \cW'$ is a topologizing subcategory of $\defff \cW.$
\end{cor}

\section{Lattice structures}

\subsection{Definitions}
We recall the following well known notions:
\begin{definition}\label{def:lattice}
A poset $P$ is called a {\em join-semilattice} if for every pair $(p,q)$ of elements of $P$ there exists a supremum $p \vee q$ (also called join). 
It is called a {\em meet-semilattice} if for every pair $(p,q)$ of elements of $P$ there exists an infimum  $p \wedge q$ (also called meet). 
Finally, $P$ is \emph{lattice} if it is both a join-semilattice and a meet-semilattice. Equivalently, a lattice is a set $P$ equipped with two binary operations $\vee$ and $\wedge:$ $P\times P \rightarrow P$ satisfying the following axioms:
\begin{enumerate}
    \item  $\vee$ is associative and commutative,
    \item  $\wedge$ is associative and commutative,

    \item  $\wedge$ and $\vee$ satisfy the following property:
    \[m \vee (m \wedge n)=m=m\wedge (m\vee n)\;\mbox{ for all } m, n \in P.\]
\end{enumerate}
A lattice is called  {\em complete} if each its subset has both a join and a meet, similarly for semilattices.
A {\em bounded} lattice is a lattice that has a greatest element (also called maximum) and a least element (also called minimum). 
\end{definition}

The concept of completeness of a lattice is self-dual, as the following standard fact shows.

\begin{lemma} \cite[Lemma 34]{Gr11}\cite[Theorem 2.31]{Da02} \label{lem: meet_join_dual}
Each subset of a poset $P$ has a join if and only if each subset of $P$ has a meet. Equivalently, a poset is a complete join-semilattice if and only if it is a complete meet-semilattice.

Explicitly, for $H \subseteq P,$ we have
$$\vee H = \wedge \{k | h \leq k, \forall h \in H\};$$
$$\wedge H = \vee \{k | k \leq h, \forall h \in H\}.$$
\end{lemma}

It is also standard that each complete lattice is bounded, see \cite[2.2, Theorem 2.31]{Da02}: the maximum is given by $\wedge \emptyset$ and the minimum is given by $\vee \emptyset$.



\begin{definition}\cite[2.16, 2.17]{Da02}\label{mor of lattices}\label{iso of lattices}
 Let $P$ and $Q$ be two lattices, then a function $f:P\rightarrow Q$ is a \emph{morphism of lattices} if for all $ m,n \in P$ one has: 
$$ f(m\vee n)= f(m) \vee f(n) \qquad and \qquad f(m\wedge n)=f(m)\wedge f(n).$$
An \emph{isomorphism of lattices}  is a bijective morphism of lattices (in which case its inverse is also an isomorphism).
\end{definition}

\begin{lemma} \cite[Proposition 2.19.(ii)]{Da02} \label{lattice_iso = poset_iso}
An map of lattices is an isomorphism of lattices if and only if it is an isomorphism of posets. In other words, each poset isomorphism of lattices preserves meets and joins.
\end{lemma}

\begin{definition} \label{def:atom}
Let $(P, \leqslant)$ be a partially ordered set with a unique minimal element 0. An  \emph{atom} is an element $a \in P$ with $a > 0$ and such that $0 \leqslant x \leqslant a$ implies $x=0$ or $x=a$. In other words, atoms are the elements that are directly above the minimal element.
\end{definition}

\subsection{Lattices of right and left weakly exact structures}
In this subsection we study a lattice structure on the class of all right (or left) weakly exact structures. These results generalise  \cite[Proposition 8.4]{HR20}.

\begin{definition}\label{set of left weakly}
We denote by $\LW(\A)$ (respectively $\RW(\A)$) the class of all left (right) weakly exact structures on $\A$. 
\end{definition}

\begin{lemma}\label{lem:semilattice}
\cite[Lemma 5.2]{BHLR} Let $\{ \mathcal{L}_{i} \}_{i \in \upomega}$ ($\{ \mathcal{R}_{i} \}_{i \in \upomega} $) be a family of left (right) weakly exact structures on $\A$. Then the intersection $\cap_{i \in \upomega} \mathcal{L}_{i}$ ($\cap_{i \in \upomega} \mathcal{R}_{i}$) is also a left (right) weakly exact structure. 
\end{lemma}

\begin{proposition}\label{thm:left-weakly-exact-form-local-lattice}
Let $\A$ be an additive category. Then $\LW(\A)$ and  $\RW(\A)$) are complete meet-semilattices.
\end{proposition}
\begin{proof}
 Let $\mathcal{L}$ and $\mathcal{L'}$ be two left weakly exact structures on $\A$. The partial order on $\LW(\A)$ is given by containment. We define the \emph{meet} given by $ \mathcal{L}\wedge \mathcal{L'} = \mathcal{L}\cap \mathcal{L'}$.  
These operations define the structure of a complete meet-semilattice  on $LW(\A)$
by Lemma \ref{lem:semilattice}.
\end{proof}
\begin{remark}
 If there exists a unique maximal left weakly exact structure $\mathcal{L}_{max}$ on $\A$, then $\LW(\A)$  is a complete lattice by Lemma \ref{lem: meet_join_dual} (similarly for $\RW(\A)$).

Likewise, any interval in the poset $\LW(\A)$ forms a complete bounded lattice.
\end{remark}

\begin{remark} The constructions in Section \ref{section:left-and-right-weakly} can be reformulated in terms of the lattices studied in this section as follows:
As stated in Proposition  \ref{prop:left-right-weakly-exact},
there is a map $$ s: \Wex(\A)  \longrightarrow \LW(\A)\times \RW(\A), \; \W\longmapsto (\mathcal{L}_{\W}, \mathcal{R}_{\W})$$
where $\mathcal{L}_{\W}:=\{ \; d \; | \; (i, d)\in \W\}$ is the class of all $\W-$cokernels or $\W-$admissible deflations and $\mathcal{R}_{\W}:=\{ \; i \; | \; (i, d)\in \W\}$ is the class of all $\W-$kernels or $\W-$admissible inflations.

Moreover, Theorem \ref{thm:right-left} shows that there is a map:
$$ g: \LW(\A)\times \RW(\A)\longrightarrow \Wex(\A) , (\mathcal{L}, \mathcal{R})\longmapsto \W_{(\mathcal{L}, \mathcal{R})}$$
where $\W_{(\mathcal{L}, \mathcal{R})}$ is formed by all short exact sequences $(i, d)$ in $\A$ with  $i\in\mathcal{R}, d\in\mathcal{L}\}.$

\end{remark}

\subsection{Lattice of weakly exact structures}
\subsubsection{Lattice of exact structures}
We recall from \cite{BHLR} that the class of exact structures on an additive category $\Ex(\A)$ forms a lattice under the following operations:
\begin{enumerate}
    \item The partial order is given by containment $\E' \subseteq \E$
    \item The meet $\wedge$ is defined by $ \E \wedge \E' = \E \cap \E' $
    \item the join $\vee_E$ is defined by
$$ \E \vee_E \E' = \bigcap \{ \cF \in \Ex(\A) \;|\; \E \subseteq \cF, \E'\subseteq \cF\}.$$ 
\end{enumerate}

\subsubsection{Lattice structure on the class of all weakly exact structures of a given additive category}
\begin{lemma}\label{lem:weakly-exact-intersection}
\cite[Lemma 5.2]{BHLR} Let $\{ \W_{i} \}_{i \in \upomega} $ be a family of weakly exact structures on $\A$. Then the intersection $\cap_{i \in \upomega} \W_{i}$ is also a weakly exact structure.    
\end{lemma}

\begin{theorem}\label{thm:weakly-exact-form-local-lattice}
Let $\A$ be an additive category and $\E_{max}$ the maximal exact structure on $\A$. Then the weakly exact structures that are included in $\E_{max}$ form a complete bounded lattice
$(\Wex({\E_{max}}), \subseteq, \wedge, \vee_{W} ).$
\end{theorem}

\begin{proof}
 It follows from Lemma \ref{lem:weakly-exact-intersection}
 that $\Wex(\A)$ forms a complete meet-semilattice: $(\Wex({\A}), \subseteq, \wedge)$
with order relation given by inclusion and meet operation given by intersection. By Lemma \ref{lem: meet_join_dual}, it is a complete (and, therefore, bounded) lattice.
\end{proof}

\begin{remark}\label{remark:weakly-exact-not-sublattice}
While the partial order and the meet coincide for $\Ex(\A)$ and $\Wex(\A)$, the join  $\vee_{E}$ 
is different from the join for weakly exact structures since we intersect over a {\em smaller set}, making the join {\em larger} when both are viewed in the poset $\Wex({\E_{max}})$:
\[ \E\vee_{W} \E' \quad \le \quad \E\vee_{E} \E'
\]
for all $\E,\E' \in \Ex(\A)$. 
In fact, in the example from Section \ref{section:example}, if we consider the exact structures $\E = \langle \alpha \rangle$ and $ \E' = \langle \gamma \rangle$, then  
$ \E\vee_{W} \E' \; = \langle \alpha, \gamma \rangle $ which is stricly smaller than $ \E\vee_{E} \E' = \langle \alpha, \gamma, \delta \rangle .$
This shows that $\Ex(\A)$ is a meet-subsemilattice of $\Wex({\E_{max}})$, but it is {\em not} a sublattice in general.
\end{remark}

We now describe the join of two weakly exact structures in a more constructive way, motivated by the sum of bifunctors:
\medskip

\begin{definition}
Let $\W_1,\W_2 \in \Wex(\E_{max})$ be two weakly exact structures contained in $\E_{max}.$ Then, $\W=\W_1 + \W_2$ is defined by taking all sums $\eta_1 + \eta_2$ with  $\eta_1 \in \W_{1}(C,A),$ and $ \eta_2 \in \W_{2}(C,A) $.
Here $\eta_1 + \eta_2 $ is the Baer sum for short exact sequences in $\E_{max}$. 
\end{definition}

\begin{proposition}\label{prop:lattice-of-weakly-exact}
Let $\W_1,\W_2$ be two weakly exact structures contained in $\E_{max}.$ Then
\begin{itemize}
    \item[(a)] $\W_1 + \W_2$ is weakly exact
    \item[(b)] $\W_1 + \W_2$ is the join $\W_1 \vee_W \W_2$ in the lattice $\Wex(\E_{max})$.
\end{itemize}
\end{proposition}
\begin{proof}
(a) It is not complicated to show that $\W_1 + \W_2$ satisfies $(E0)$, $(E2)$ and their dual $(E0, E2)^{op}$. Moreover, consider
$\alpha:  \in \W(C,A)$ and 
$\beta: \in \W(F,D)$.
Then there exist $\alpha_1 \in \W_1(C,A)$, $\alpha_2 \in \W_2(C,A)$, $\beta_1 \in \W_1(F,D)$ and $\beta_2 \in \W_2(F,D)$ such that $\alpha=\alpha_1+\alpha_2$ and $\beta=\beta_1+\beta_2$, hence 
\[\alpha \oplus \beta = (\alpha_1+\alpha_2)\oplus(\beta_1+\beta_2).\]
Since $\W_1$ and $\W_2$ are closed under direct sums, we get that $\alpha \oplus \beta = (\alpha_{1} \oplus \beta_{1}) + (\alpha_{2} \oplus \beta_{2}) \in \W(C\oplus F,A\oplus D) \subseteq \W$. Therefore $\W$ is closed under direct sums and it is a weakly exact structure. \medskip

 (b) Recall that the join $\W_1 \vee_W \W_2$ is the smallest  weakly exact structure on $\A$ containing both $\W_1$ and $\W_2$. 
We have that $\W_1 \subset \W_1 + \W_2$ since $\eta_1 = \eta_1 + 0 \in \W_1 + \W_2$ for any $\eta_1  \in \W_1.$ Likewise for $\W_2$, so  $\W_1 + \W_2$ contains both $\W_1$ and $\W_2$, hence by definition of the join, $\W_1 \vee_W \W_2 \subseteq \W_1 + \W_2.$\\
To show the converse inclusion, let $\W$ be any weakly exact structure containing both $\W_1$ and $\W_2$. Since $\W$ satisfies the direct sum property (S), we have $\eta_1 \oplus \eta_2 \in \W$ for all $\eta_1 \in \W_1, \eta_2 \in \W_2.$
By definition of Baer sum and property (E2) and (E2)$^{op}$ for $\W$ we have $\eta_1 + \eta_2 \in \W$. 
This shows  $ \W_1 + \W_2 \subset \W $ for {\em all} $\W$ containing both $\W_1$ and $\W_2$, so this also holds for the smallest one (their intersection) : 
$\W_1 + \W_2 \subseteq \W_1 \vee_W \W_2.$
\end{proof}

\begin{proposition}\label{prop:atoms}
Let $\alpha$ be an Auslander-Reiten sequence in $\A$, and denote by $\E_{\alpha}=\{ X \oplus Y \; | \; X \in \E_{min}, Y \in add(\alpha) \}$  the (weakly) exact structure generated by $\alpha$. Then $\E_{\alpha}$ is an atom of both lattices $(\Ex({\A}), \subseteq, \wedge, \vee_E)$ and $(\Wex({\A}), \subseteq, \wedge, \vee_W)$.
\end{proposition}
\begin{proof}
This corresponds to the well-known property that the Auslander-Reiten sequences lie in the socle of the bifunctor $\Ext^1_\A(-,-)$, that is, multiplication with morphisms does not generate any new non-split sequences. 
\end{proof}

\subsection{Lattice of additive sub-bifunctors of ${\Ext}^{1}_{\A}$}
In Section \ref{section:subfunctors}, we discussed additive sub-bifunctors of $\Ext^{1}_{\A}:=\bE_{max}=\Ext^1_{\E_{max}}$ and  closed additive sub-bifunctors,  and we denote these classes respectively by $\Bf(\bE_{max})$ and $\Cbf(\A)$. In this section, we construct lattice structures of both classes.

\begin{theorem} \label{theorem:latticeBiFun}
The additive sub-bifunctors of $\bE_{max}$ form a lattice \[(\Bf(\bE_{max}), \leq, \wedge, \vee_{bf}).\]
\end{theorem}
\begin{proof}
For $F, F' \in \Bf(\bE_{max})$, we write $F \leq F'$ if $F$ is a sub-bifunctor of $F'$. \\
The meet  of $F$ and $F'$ is given by the sub-bifunctor $F \wedge F'$ of $\bE_{max}$ satisfying 
\[(F \wedge F')(C,A)=F(C,A) \cap F'(C,A)  \mbox { for all } A, C \in \A. \] 
The join is given by the sub-bifunctor $F+F' \; =\; F \vee_{bf} F'$ of $\bE_{max}$ satisfying 
\[(F \vee_{bf} F')(C,A)=F(C,A) + F'(C,A)  \mbox { for all } A, C \in \A, \] 
where the sum is the sum of abelian subgroups of $\bE_{max}(C,A)$.
Since $\Bf(\bE_{max})$ has a maximal element $\bE_{max}$, the statement follows from Lemma \ref{lem: meet_join_dual}.
\end{proof}

\subsubsection{Lattice of closed additive sub-bifunctors}

As discussed in Proposition \ref{prop:bijections}, for any additive category $\A$ there is a bijection between exact structures and closed additive sub-bifunctors of $\bE_{max}$. We already  know that the exact structures form a lattice \cite[Theorem 5.3]{BHLR}. In this section we define a lattice structure on  the class $\Cbf(\A)$ of closed additive sub-bifunctors of $\bE_{max}$.

\begin{lemma} \cite[corollary 1.5]{DRSS}\label{lem:intersection-closed-sub-bifunctor}
Consider a family $\{F_{i}\}_{i \in I}$ of closed sub-bifunctors of $\bE_{max}$. Then the intersection $\cap_{i \in I} F_{i}$ is a closed sub-bifunctor of $\bE_{max}$bifunctor, given by $\{\cap F_{i}\}(C,A)=\cap \{F_{i}(C,A)\}$ on objects.
\end{lemma}

\begin{remark}
If $F$ and $F'$ are closed bifunctors in $\Cbf(\A)$ then their sum $F+F'$ is the sub-bifunctor of $\bE_{max}$
given by $\{F+F'\}(C,A)=F(C,A)+F'(C,A)$ on objects. Note that the sum of closed sub-bifunctors is not always closed.
\end{remark}

\begin{theorem} \label{theorem:latticeCbf}
For  an additive category $\A$, the closed additive sub-bifunctors of $\bE_{max}$ form a complete bounded lattice $(\Cbf(\A), \leq, \wedge, \vee_{cbf})$. 
\end{theorem}
\begin{proof}
The meet is defined by $F \wedge F' = F \cap F'$ and 
Lemma \ref{lem:intersection-closed-sub-bifunctor} ensures that $\Cbf(\A)$ forms a complete meet-semilattice. Lemma \ref{lem: meet_join_dual} turns it into a complete lattice, which is bounded by $\bE_{min}$ below and $\bE_{max}$ above.
\end{proof}

\begin{remark}
The closed sub-bifunctors (\Cbf($\A$), $\leq$) form a subposet of \\ $(\Bf(\bE_{max}), \leq)$. However,  $(\Cbf(\A, \leq, \wedge, \vee_{cbf})$ is not a sublattice of \\ $\Bf(\bE_{max}), \leq, \wedge, \vee_{bf})$ because their joins are different. In fact, for $F, F' \in \Cbf(\A),$ the join  
$F \vee_{bf} F'= F+F'$ is not necessarily closed. As discussed in Remark \ref{remark:weakly-exact-not-sublattice}, the join of $<\alpha>$ with $<\gamma>$ in $\Bf(\bE_{max})$ is $<\alpha,\gamma>$ which is not closed. The join of  $<\alpha>$ with $<\gamma>$, in $\Cbf(\A)$ is $<\alpha,\gamma,\delta>$.  
In general, for $F, F' \in \Cbf(\A)$ we have that $F \vee_{bf} F' \leq F \vee_{cbf} F'$.
\end{remark}

\subsection{Lattice of bimodules over the Auslander algebra}
We return now to the study of the bimodule $\B$ over the Auslander algebra $A$ defined in Section \ref{sec:bimodules}.
As is the case for any module over a ring, recall that the set $\Bim(\B)$ of sub-bimodules of $\B$
forms a complete bounded modular lattice $$(\Bim(\B), \leq, {\wedge}_{Bim}, {\vee}_{Bim}),$$
where the meet is given by intersection and the join is given by the sum $N + N'$ of sub-bimodules.

\begin{definition}\label{def:evaluation map}
An element $N\in \Bim(\B)$ is said to be a \emph{closed} bimodule if  there exists a \emph{closed} sub-bifunctor $F$ of  $\Ext^{1}_{\E_{max}}$ such that $Ev_X(F)=N$ where 
\[Ev_X: \Cbf(\mathcal{A}) \longrightarrow \Bim(\B), \quad F \mapsto F(X,X)\]
is the evaluation at the object $X \in \A.$
\end{definition}
\begin{lemma}
The intersection of two closed sub-bimodules of $\B$ is again closed.
\end{lemma}
\begin{proof}
Let $N$ and $P$ be two closed sub-bimodules of $\B$ such that there exists two closed sub-bifunctors $F$  and $G$ satisfying $\Phi(F)=N$ and $\Phi(G)=P$. We consider the sub-bifunctor $H$ of $\Ext^1_{\E_{max}}$ given by the meet of $F\wedge G= H$. By Lemma \ref{lem:intersection-closed-sub-bifunctor}, H is closed. Since $N\cap P = F(X,X)\cap G(X,X)= H(X,X),$ the intersection is  a closed sub-bimodule of $\B$.
\end{proof}

\begin{theorem} The subset $\Cbim(\B)$ of closed sub-bimodules of $\B$ forms  a complete (and, in particular, bounded) lattice $(\Cbim(\B),\subseteq, \cap, {\vee}_{Cbim}).$

\end{theorem}
\begin{proof}
This class is a poset ordered by inclusion. 
Let $\{N_{\lambda} \}_{\lambda \in \Lambda}$ be a family of weakly exact structures in $\Cbim(B)$. Their meet is given by  the associative, commutative intersection of modules  $\underset{\lambda \in \Lambda} \cap  N_{\lambda}$. Since the lattice has a minimal element $0$ and a maximal element $\B$, it is a bounded lattice. By Lemma \ref{lem: meet_join_dual}, it is a complete lattice.
\end{proof}


In the setting of this subsection, the bimodule $\B=\bE_{max}(X,X)$ is finite-dimensional, thus $\B$ and all of its submodules have a non-zero socle. We know from Proposition \ref{prop:atoms} that the Auslander-Reiten sequences lie in the socle of the bimodule $B$, and since all non-projective objects admit an Auslander-Reiten sequence in $\A$ ending there, 
one can derive that the socle is precisely formed by all Auslander-Reiten sequences in $\A$. Based on Auslander's concept of defects, Enomoto shows in \cite{En18} that the lattice $\Cbim(\B)$ is an atomic lattice, in fact it is a boolean lattice determined by its atoms, the Auslander-Reiten sequences in $\A$ (see also \cite[Theorem 2.26]{FG}).

Reformulated in module-theoretic terms, that means that the closed sub-bimodules of $B=\bE_{max}(X,X)$ are uniquely determined by their socle, and for every choice of elements in the socle, there is a unique closed sub-bimodule of $B$ having precisely these elements as its socle. If the socle is formed by a set $S$ of Auslander-Reiten sequences, we can thus denote by $\bE(S)$ the subbimodule of $B$ determined by $S$. For all elements $\sigma \in S$, denote by $\bE_\sigma $ the bimodule corresponding to the exact structure $\E_\sigma$ introduced in Proposition \ref{prop:atoms}.  
Since the lattice $\Cbim(\B)$ is atomic, we conclude that
\[ \bE(S) = \bigvee_{\sigma \in S}\bE_\sigma .\]
There may be several submodules of $\B$ with the same socle $S$, but only one of them is closed. 
As explained in the proof of \cite[Theorem 2.26]{FG}, this closed submodule with  socle $S$ corresponds to a Serre subcategory $\S$ generated by the simple objects contained in the set $S$. All other submodules of $\B$ with socle $S$ correspond to certain subcategories of $\S,$ but only the closed one is given by the abelian length category formed by all extensions of its simple objects. In other words, $\bE(S)$ is maximal, so
we derive the following result which  is illustrated nicely in the example in Section \ref{section:example}. It is also shown independently for Nakayama algebras in \cite[Theorem 6.9]{BHT}.

\begin{proposition}\label{prop:socle}
For every set  $S$ of Auslander-Reiten sequences, the closed bimodule $\bE(S)$ of $\B$ introduced above  is the maximal submodule of $\B$ whose socle is $S$.
\end{proposition}


\subsection{Lattice of topologizing subcategories}


Topologizing subcategories of an abelian category $\cc$ form a complete lattice. The order is given by the canonical inclusion of categories and the meet is given by the usual intersection. This is a complete semi-lattice and, therefore, it has a canonical join operation upgrading it to a complete lattice. It is straightforward to check from the definitions that the join is given by the closure of the union by finite direct sums: $$\bigvee: Top(\cc)\times Top(\cc) \rightarrow Top(\cc), \quad (T, T')\longmapsto \oplus \{T\cup T'\}.$$ Since this lattice has a canonical minimal element, it is moreover bounded.
By definition, each Serre subcategory of an abelian category is topologizing. Thus, Serre subcategories form a subposet of the lattice of topologizing subcategories. By similar arguments this subposet admits a lattice structure, with the join given by the closure of the union by finite extensions. Since the closure of the union by finite direct sums is, in general, not extension-closed, the join of Serre subcategories in the lattice of topologizing subcategories is different from their join in the lattice of Serre subcategories. In other words, the lattice of Serre subcategories is a subposet, but not a sublattice of the lattice of all topologizing subcategories.

 Given a topologizing subcategory $\cc$ of the category $\coh(\ca)$, its topologizing subcategories in the sense of definition \ref{def:topologizing} form a lattice, which is an interval in the lattice of all topologizing subcategories in $\coh(\ca).$ Serre subcategories of $\cc$ form a lattice, which is an interval in the lattice of all Serre subcategories in $\coh(\ca).$ It is a subposet, but not a sublattice of the lattice of topologizing subctegories of $\cc.$

We formulate this observation explicitly in the case of the categories of defects of weakly extriangulated structures:


\begin{proposition}
Let $\A$ be an essentially small category and $(\mathbb{W}, \mathfrak{s})$ a weakly extriangulated structure on it, then the topologizing subcategories of $\defff \mathbf{W}$ form a complete lattice
\[(\mathbf{Top}(\mathbf{W}), \subseteq,\bigcap, \bigvee).\] Serre subcategories of $\defff \mathbf{W}$ also form a lattice, which is a subposet, but not a sublattice of $(\mathbf{Top}(\mathbf{W})).$
\end{proposition}

\subsection{Lattices of extriangulated and weakly extriangulated substructures}

Let $\A$ be an essentially small additive category. We consider the class of all weakly extriangulated structures on $\A$.


\begin{lemma}\label{Intersection of weakly extriangulated}

Let $\{ {\mathbb{W}}_{i} \}_{i \in \upomega} $ be a family of weakly extriangulated structures on $\A$. Then the intersection $\cap_{i \in \upomega} {\mathbb{W}}_{i}$ is also a weakly extriangulated structure.    
\end{lemma}

\begin{proof}
Similar to Lemma 5.2 of \cite{BHLR}.
\end{proof}

\begin{theorem}
Let $(\cA, \mathbb{W}, \mathfrak{s})$ be a weakly extriangulated category. Then all its weakly extriangulated substructures  form a complete lattice:
$$(\mathbf{WET}(\A), \leq, \bigwedge, \bigvee)$$

\end{theorem}
\begin{proof}
We consider the set $\mathbf{WET}(\A)$ of all the additive sub-bifunctors of $\mathbb{W}$ on the essentially small category $\A$. They are ordered by 
\[W\leq W' \Longleftrightarrow W(C,A) {\subseteq}_{Ab} W'(C,A) \mbox{ for all } A,C \in \A\]
that is, $W(C,A)$ is a subgroup of $W'(C,A)$ for every pair of objects in $\A$. 
It follows from \ref{Intersection of weakly extriangulated} that $(\mathbf{WET}(\A), \leq, \bigwedge)$ is a meet semi-lattice with the meet $(W \bigwedge W')(C, A)= W(C, A)\cap W'(C, A) ,\forall A, C\in \A$, by using the intersection of abelian groups. By Lemma \ref{Intersection of weakly extriangulated}, it is a complete meet semi-lattice.
By Lemma \ref{lem: meet_join_dual}, it is a complete lattice 
 with minimal element given by the split weakly extriangulated structure $\mathbb{W}_{min}$. 
\end{proof}

\begin{cor}
Let $(\cA, \mathbb{E}, \mathfrak{s})$ be an extriangulated category. Then all the additive sub-bifunctors of $\mathbb{E}$ form a complete (in paricular, bounded) lattice.
\end{cor}


\subsection{Isomorphims of lattices}

\begin{theorem}\label{iso 2}
Let $\A$ be an additive category. The map $\Phi: \W \mapsto \Ext^1_\W(-,-)$ induces a \emph{lattice isomorphism} \[(\Wex(\E_{max}),\subseteq, \cap, \vee_{\W}) \; \cong \;(\Bf(\bE_{max}), \leq, \wedge, \vee_{bf}).\]
\end{theorem}

\begin{proof}
We have already shown in Proposition \ref{prop:bijections} that $\Phi$
is an isomorphism of posets. Thus, it is a lattice isomorphism by Lemma \ref{lattice_iso = poset_iso}.
\end{proof}

\begin{theorem}\label{iso 3}
Consider the setting of an additively finite category $\A$ as in Section \ref{sec:bimodules} and the bimodule $\B$ over the Auslander algebra $A$ defined there.
Then the evaluation map  yields an isomorphism of lattices
\[Ev_X: \Bf(\bE_{max}) \longrightarrow \Bim(\B), \quad F \mapsto F(X,X)\]
\end{theorem}
\begin{proof} It is easy to show that the map  $Ev_X$ is well defined,  morphism of posets, and morphism of lattices.
To construct an inverse map, consider a sub-bimodule $W$ of $_AB_A$. 
As $A =  \bigoplus_{i,j}\Hom(X_i,X_j)$ with $X=X_1 \oplus \cdots \oplus X_n$ we
can decompose $W$ as $W= \bigoplus_{i,j} e_j  W e_i$.
By additive extension to arbitrary direct sums of the $X_i$ one can then define a sub-bifunctor $\W$ of $\bE_{max}$ with $\W(X,X)= W$. 
\end{proof}

\begin{cor}\label{three large lattices} If $\mathcal{A}$ is an additively finite, Hom-finite Krull-Schmidt category then the three lattice structures we defined on $\Wex(\A)$, $\Bf(\A)$ and $\Bim(B)$ are isomorphic.
  
\end{cor}


\begin{theorem}\label{bijection-lattices}
Let $\A$ be an additive category. The map $\Phi: \E \mapsto \Ext^1_\E(-,-)$ induces a \emph{lattice isomorphism} between $(\Ex(\A), \subseteq, \cap, \vee)$ and $(\Cbf(\A), \leq, \wedge, \vee)$.
\end{theorem}

\begin{proof}
Same as for Theorem \ref{iso 2}. 
\end{proof}

\begin{theorem}\label{small iso 2} 
If $\mathcal{A}$ is an additively finite, Hom-finite Krull-Schmidt category then the two lattices $(\Cbf(\A), \leq, \wedge, \vee_{Cbf})$ and $(\Cbim(B), \subseteq, \cap, \vee_{Cbim})$ are isomorphic.
\end{theorem}
\begin{proof}
As already verified in Theorem \ref{iso 3}, the evaluation map $Ev_X$ preserves the order and the meet-semilattice structure. By definition, its restriction to closed sub-bifunctors is a poset isomorphism with the lattice $(\Cbim(B), \subseteq, \cap, \vee_{Cbim})$. Thus, it is a lattice isomorphism by Lemma \ref{lattice_iso = poset_iso}. 
\end{proof}

\begin{cor} If $\mathcal{A}$ is an additively finite, Hom-finite Krull-Schmidt category then the three lattice structures defined above on $\Ex(\A)$, $\Cbf(\A)$ and $\Cbim(\B)$ are isomorphic.
\end{cor}


\begin{proposition}
Let $(\cA, \mathbb{W}, \mathfrak{s})$ be a weakly extriangulated category. Then there is a lattice isomophism between the lattice of additive sub-bifunctors of $\mathbb{W}$ and the lattice of topologizing subcategories of $\defff \mathbb{W}.$
\end{proposition}

\begin{proof}
The proof of \cite[Theorem B]{Enomoto3}, with Step 3 removed, applies.
\end{proof}

\begin{cor}
Let $\cW$ be a weakly exact structure on $\ca$. Then there is a lattice isomorphism between the interval $[\cW^{\add}, \cW]$ in the lattice of weakly exact structures on $\ca$ and the lattice of topologizing subcategories of $\defff \mathbb{W}.$ 
\end{cor}

\begin{cor}\label{iso4}
When the category $\ca$ admits a unique maximal weakly exact structure $\cW^{\max},$ the lattice of weakly exact structures on $\ca$ is isomorphic to the lattice of topologizing subcategories of $\defff \mathbb{W}^{\max}.$ 
\end{cor}
In particular we get the following summarising result:
\begin{cor}\label{all in one}Let $\A$ be weakly idempotent complete essentially small additive category, then the following three lattices are isomorphic:
 \[\Wex(\A) \overset\sim\to \Bf(\A) \overset\sim\to \defff \mathbf{\mathcal{E}}^{\max}.\]
 If $\A$ is moreover additively finite, Hom-finite, and Krull-Schmidt, they are also isomorphic to the lattice $\Bim(B).$
\end{cor}

Note that when $\ca$ is idempotent complete, we can use arguments from \cite{En18, Enomoto2, FG}  instead. In particular, this approach would give another proof of the existence of $\cW^{\max}$ in this generality.

\end{document}